\documentclass[11pt]{amsart}
\usepackage[T1]{fontenc}

\newcommand{\EE}{\mathbb{E}}
\newcommand{\FF}{\mathbb{F}}
\newcommand{\GG}{\mathbb{G}}

\newcommand{\NN}{\mathbb{N}}
\newcommand{\PP}{\mathbb{P}}

\newcommand{\ZZ}{\mathbb{Z}}

\newcommand{\cB}{\mathcal{B}}
\newcommand{\cC}{\mathcal{C}}

\newcommand{\cE}{\mathcal{E}}

\newcommand{\cG}{\mathcal{G}}
\newcommand{\cH}{\mathcal{H}}

\newcommand{\cM}{\mathcal{M}}

\newcommand{\cU}{\mathcal{U}}

\newcommand{\cW}{\mathcal{W}}

\newcommand{\fA}{\mathfrak{A}}
\newcommand{\fb}{\mathfrak{b}}

\newcommand{\fe}{\mathfrak{e}}
\newcommand{\ff}{\mathfrak{f}}

\newcommand{\fc}{\mathfrak{c}}
\newcommand{\fg}{\mathfrak{g}}

\newcommand{\fh}{\mathfrak{h}}

\newcommand{\fn}{\mathfrak{n}}

\newcommand{\fsl}{\mathfrak{sl}}
\newcommand{\ft}{\mathfrak{t}}

\newcommand{\fu}{\mathfrak{u}}

\newcommand{\msrke}{\mathsf{rk}_p}

\newcommand{\dact}{\boldsymbol{.}}
\newcommand{\lra}{\longrightarrow}

\DeclareMathOperator{\ad}{ad}

\DeclareMathOperator{\Aut}{Aut}

\DeclareMathOperator{\Char}{char}

\DeclareMathOperator{\cx}{cx}

\DeclareMathOperator{\Der}{Der}

\DeclareMathOperator{\Ext}{Ext}

\DeclareMathOperator{\Gr}{Gr}

\DeclareMathOperator{\HH}{H}
\DeclareMathOperator{\height}{ht}

\DeclareMathOperator{\Hom}{Hom}

\DeclareMathOperator{\im}{im}
\DeclareMathOperator{\id}{id}

\DeclareMathOperator{\Lie}{Lie}

\DeclareMathOperator{\modd}{mod}

\DeclareMathOperator{\mstor}{\mathsf{Tor}}

\DeclareMathOperator{\msP}{\mathsf{P}}

\DeclareMathOperator{\PSL}{PSL}

\DeclareMathOperator{\pt}{Pt}

\DeclareMathOperator{\Rad}{Rad}

\DeclareMathOperator{\rk}{rk}

\DeclareMathOperator{\SL}{SL}

\DeclareMathOperator{\Spec}{Spec}

\DeclareMathOperator{\Proj}{Proj}

\usepackage{pictex}
\usepackage{amscd}
\usepackage{latexsym}
\usepackage{amssymb}
\usepackage{eucal}
\usepackage{eufrak}
\usepackage{verbatim}

\numberwithin{equation}{section}

\newtheorem{Theorem}{Theorem}[section]
\newtheorem{Lemma}[Theorem]{Lemma}
\newtheorem{Corollary}[Theorem]{Corollary}
\newtheorem{Proposition}[Theorem]{Proposition}

\theoremstyle{Theorem}
\newtheorem{Thm}{Theorem}[subsection]
\newtheorem{Lem}[Thm]{Lemma}
\newtheorem{Prop}[Thm]{Proposition}
\newtheorem{Cor}[Thm]{Corollary}

\newtheorem*{thm*}{Theorem A}
\newtheorem*{thm**}{Corollary}
\newtheorem*{thm***}{Theorem B} 

\theoremstyle{remark}
\newtheorem*{Remark}{Remark}
\newtheorem*{Remarks}{Remarks}
\newtheorem*{Definition}{Definition}
\newtheorem*{Example}{Example}
\newtheorem*{Examples}{Examples}

\numberwithin{equation}{section}

\begin{document}

\title{Finite group schemes of $p$-rank $\le 1$}

\author[Hao Chang and Rolf Farnsteiner]{Hao Chang \lowercase{and} Rolf Farnsteiner}
\address{Mathematisches Seminar, Christian-Albrechts-Universit\"at zu Kiel, Ludewig-Meyn-Str. 4, 24098 Kiel, Germany}
\email{chang@math.uni-kiel.de, rolf@math.uni-kiel.de}
\date{\today}

\begin{abstract} Let $\cG$ be a finite group scheme over an algebraically closed field $k$ of characteristic $\Char(k)=p\ge 3$. In generalization of the familiar notion from the modular representation 
theory of finite groups, we define the $p$-rank $\msrke(\cG)$ of $\cG$ and determine the structure of those group schemes of $p$-rank $1$, whose linearly reductive radical is trivial. The most 
difficult case concerns infinitesimal groups of height $1$, which correspond to restricted Lie algebras. Our results show that group schemes of $p$-rank $\le 1$ are closely related to those being 
of finite or domestic representation type. \end{abstract}

\maketitle

\section*{Introduction}
Let $(\fg,[p])$ be a finite-dimensional restricted Lie algebra over an algebraically closed field $k$ of characteristic $\Char(k)\!=\!p\!>\!0$. In the representation
theory of $(\fg,[p])$ the closed subsets of the nullcone
\[ V(\fg) :=\{ x \in \fg \ ; \ x^{[p]}=0\}\]
play an important role. The elements of its projectivization $\PP(V(\fg)) \subseteq \PP(\fg)$ can be  construed as one-dimensional subalgebras
of $\fg$ which are annihilated by the $p$-map. More generally, one can consider for each $r \in \NN$ the closed subset
\[ \EE(r,\fg) := \{ \fe \in \Gr_r(\fg) \ ; \ [\fe,\fe]=(0), \, \fe \subseteq V(\fg)\}\]
of the Grassmannian $\Gr_r(\fg)$ of $r$-planes in $\fg$. These projective varieties were first systematically studied by Carlson-Friedlander-Pevtsova in \cite{CFP15}.
The elements of the variety $\EE(r,\fg)$ are analogs of $p$-elementary abelian subgroups and it is therefore natural to explore their utility for the representation
theory of $(\fg,[p])$. 

In analogy with the modular representation theory of finite groups, we define the \textit{$p$-rank} of $\fg$, via
\[\msrke(\fg) := \max \{ r \in \NN_0 \ ; \ \EE(r,\fg) \ne \emptyset\}.\]
Restricted Lie algebras of $p$-rank $\msrke(\fg)=0$, which are the analogs of those finite groups, whose group algebras are semi-simple, were determined in early work by Chwe \cite{Ch65}, 
who showed that $\msrke(\fg)=0$ if and only if $\fg$ is a torus (i.e., $\fg$ is abelian with bijective $p$-map). In view of \cite{Ho54}, this is equivalent to the semi-simplicity of the
restricted enveloping algebra $U_0(\fg)$ of $\fg$. 

In recent work \cite{Fa14}, invariants for certain $(\fg,[p])$-modules were introduced that turned out to be completely determined on $\EE(2,\fg)$. This raised the question concerning properties of this variety, the most basic one pertaining to criteria for $\EE(2,\fg)$ being non-empty, that is, $\msrke(\fg)\!\ge\!2$. Contrary to Lie algebras of $p$-rank $0$, the answer somewhat depends
on the characteristic of $k$. 

\bigskip

\begin{thm*} Suppose that $p\!\ge\!5$. If $\fg$ affords a self-centralizing torus, then $\msrke(\fg)=1$ if and only if
\[ \fg/C(\fg) \cong \fsl(2), \fb_{\fsl(2)}, \fb_{\fsl(2)}^{-1},\]
with the center $C(\fg)$ of $\fg$ being non-zero in the latter case. \end{thm*}

\bigskip
\noindent
Here $\fb_{\fsl(2)}$ denotes the standard Borel subalgebra of $\fsl(2)$, while $\fb_{\fsl(2)}^{-1}$ is a one-dimensional non-split abelian extension of $\fb_{\fsl(2)}$.
For $p\!=\!3$, or results are not as definitive, as more algebras can occur if $C(\fg)\ne (0)$. For centerless or perfect Lie algebras, however, we have
complete results for $p\!\ge\!3$.

Our main techniques are based on invariants of $\fg$ that are derived from generic properties of root space decompositions relative to tori of maximal dimension.  
As usual, these methods are more effective for Lie algebras of algebraic groups, where a complete classification is fairly straightforward, cf.\ Theorem
\ref{AL1}.  In the general case, our technical assumption concerning the existence of a self-centralizing torus mainly rules out solvable Lie algebras affording 
at most one root: As we show in Section \ref{S:ER}, a Lie algebra of $p$-rank $\msrke(\fg)\!\le\!1$ which has at least two roots, always possesses such a torus. 
This readily implies:

\bigskip

\begin{thm**} Suppose that $p\!\ge\!5$. If $\fg$ affords a torus of maximal dimension, whose set of roots has at least three elements, then $\msrke(\fg)\!\ge\!2$.  \end{thm**}

\bigskip
\noindent
For $p\ge 3$, the analogous problem of determining those finite groups $G$, whose $p$-elementary abelian subgroups all have $p$-rank $\le 1$ leads to the consideration of groups, whose 
Sylow-$p$-subgroups are cyclic: Quillen's Dimension Theorem readily implies that the complexity $\cx_G(k)$ of the trivial $G$-module is bounded by $1$, cf.\ \cite[Theorem]{AE81}. In view of 
\cite[(XII.11.6)]{CE56}, this is equivalent to the Sylow-$p$-subgroups being cyclic. The interested reader may consult Brauer's papers \cite{Br76,Br79} for further information concerning the 
structure of such finite groups. 

In view of Higman's classical result \cite{Hi54}, the determination of finite groups of $p$-rank $\le 1$ is equivalent to finding those finite groups, whose group algebras have finite 
representation type. As we show in Section \ref{S:FG}, there is a similar connection for finite group schemes, once one also allows group schemes of domestic representation type. This observation
rests on the following result, whose proof employs Theorem A in order to give the following characterization of finite group schemes $\cG$ of $p$-rank $1$ and with trivial largest linearly reductive normal subgroup $\cG_{\rm lr}$:

\bigskip

\begin{thm***} Suppose that $p\ge 3$ and let $\cG$ be a finite group scheme such that $\cG_{\rm lr}=e_k$. If $\msrke(\cG)=1$, then one of the following alternatives occurs:
\begin{enumerate}
\item[(a)] $\cG = \cG_{\rm red}$, and the finite group $\cG(k)$ has $p$-rank $1$ and $O_{p'}(\cG(k))=\{1\}$.
\item[(b)] There is a binary polyhedral group scheme $\tilde{\cG} \subseteq \SL(2)$ such that $\cG \cong \PP(\SL(2)_1\tilde{\cG})$.
\item[(c)] $\cG = \cU\!\rtimes\!\cG_{\rm red}$, where $\cU$ is V-uniserial of height $\height(\cU)\ge 2$ and $\cG(k)$ is cyclic and such that
$p\!\nmid\!{\rm ord}(\cG(k))$.
\item[(d)] $\cG = ((\cW_n)_1\!\rtimes\!\GG_{m(r)})\!\rtimes\!\cG_{\rm red}$, where $\cG(k)$ is abelian and $p\!\nmid\!{\rm ord}(\cG(k))$. \end{enumerate} \end{thm***} 

\bigskip
\noindent
Finite group schemes $\cG$ of $p$-rank $0$ were determined by Nagata, who showed that the infinitesimal and reduced constituents of such groups are diagonalizable
and of order prime to $p$, respectively.

\emph{Throughout this paper, all vector spaces are assumed to be finite-dimensional over a fixed algebraically closed field $k$ of characteristic $p\!\ge\!3$}. The reader
is referred to \cite{SF} for basic facts concerning restricted Lie algebras and their representations. 

\bigskip
\noindent
\textbf{Acknowledgments.} This paper was mainly written during a two-month visit of the second author to the Collaborative Research Center 701 of the University of 
Bielefeld. He would like to take this opportunity and thank Henning Krause and his research team for their hospitality. 

The authors are indebted to Claus Michael Ringel for sharing his manuscript \cite{Ri16} with them. 

\bigskip

\section{Preliminaries}
\subsection{Root space decompositions}
The proof of the following fundamental result, which is based on the Projective Dimension Theorem \cite[(I.7.2)]{Ha}, is analogous to the arguments employed by
C.M.\ Ringel in his description of the elementary modules of the $3$-Kronecker quiver, cf.\ \cite[(3.2)]{Ri16}.

\bigskip

\begin{Lem} \label{RSD1} Let $\fg$ be a Lie algebra $U,V \subseteq \fg$ be subspaces such that 
\begin{enumerate}
\item[(a)] $\dim_kU=2$, and
\item[(b)] $\dim_k[U,V] \le \dim_kV$. \end{enumerate}
Then there exist $u\in U\!\smallsetminus\!\{0\}$ and $v \in V\!\smallsetminus\!\{0\}$ such that $[u,v]=0$. \end{Lem}

\bigskip
\noindent
Let $\fg$ be a Lie algebra with root space decomposition
\[ \fg = \fh\!\oplus\!\bigoplus_{\alpha \in R} \fg_\alpha\]
relative to some Cartan subalgebra $\fh \subseteq \fg$. 

\bigskip

\begin{Cor} \label{RSD2} Let $\fg_\alpha$ be a root space of maximal dimension. If $\beta \in R\!\smallsetminus\!\{-\alpha\}$ is such that $\dim_k\fg_\beta \ge 2$,
then there are $x \in \fg_{\alpha}\!\smallsetminus\!\{0\}$ and $y \in \fg_\beta\!\smallsetminus\!\{0\}$ such that $[x,y]=0$. \end{Cor}

\begin{proof} By assumption, there is a subspace $U\subseteq \fg_{\beta}$ of dimension $2$. By choice of $\alpha$ and $\beta$, we have 
\[ \dim_k [U,\fg_\alpha]\le \dim_k \fg_{\alpha+\beta} \le \dim_k\fg_\alpha,\] 
so that our assertion follows from Lemma \ref{RSD1}. \end{proof} 

\bigskip
\noindent
Let $(\fg,[p])$ be a restricted Lie algebra. We set 
\[ \mu(\fg):= \max \{ \dim_k\ft \ ; \ \ft \subseteq \fg \ \text{torus}\}\]
and consider 
\[ \mstor(\fg) := \{ \ft \subseteq \fg \ ; \ \ft \ \text{torus}, \ \dim_k \ft = \mu(\fg)\}.\]
Let $\ft \in \mstor(\fg)$ be a torus of dimension $\mu(\fg)$. By general theory, the centralizer $C_\fg(\ft)$ of $\ft$ in $\fg$ is a Cartan subalgebra, and
there results the root space decomposition
\[ \fg = C_\fg(\ft) \oplus \bigoplus_{\alpha \in R_\ft} \fg_\alpha\]
$\fg$ relative to $\ft$. The set $R_\ft \subseteq \ft^\ast\!\smallsetminus\!\{0\}$ is called the set of roots of $\fg$. 

We denote by
\[ \rho(\fg,\ft) := \max_{\alpha \in R_\ft} \dim_k\fg_\alpha \]
the maximal dimension of the root spaces and let
\[ r(\fg,\ft):=|R_\ft|\]
be the number of roots. It turns out that these data do not depend on the choice of $\ft$.

\bigskip

\begin{Lem} \label{RSD3} Let $(\fg,[p])$ be a restricted Lie algebra. Then there exist $\rho(\fg), r(\fg) \in \NN$ such that
\[ \rho(\fg,\ft) = \rho(\fg) \ \text{and} \ r(\fg,\ft)=r(\fg)\]
for all $\ft \in \mstor(\fg)$. \end{Lem}

\begin{proof} Let $\ft,\ft' \in \mstor(\fg)$, then \cite[(4.3)]{Fa04} implies that 
\[ \rho(\fg,\ft) = \rho(\fg,\ft') \ \text{as well as} \ r(\fg,\ft) = r(\fg,\ft'),\]
as desired. \end{proof}  

\bigskip

\subsection{Generically toral Lie algebras}
Let $(\fg,[p])$ be a restricted Lie algebra. We say that $(\fg,[p])$ is {\it generically toral}, provided there is a self-centralizing torus $\ft \subseteq \fg$. 
In view of \cite[(3.8), (7.6)]{Fa04}, every torus $\ft \in \mstor(\fg)$ of a generically toral Lie algebra is self-centralizing and the set $\bigcup_{\ft \in \mstor(\fg)}\ft$
lies dense in $\fg$.

If $\fg$ is generically toral, then its center $C(\fg)$ coincides with $\bigcap_{\ft \in \mstor(\fg)}\ft$. In particular, $C(\fg)$ is a torus. 

\bigskip 

\begin{Lem} \label{Gt1} Let $(\fg,[p])$ be generically toral, $\fc \subseteq C(\fg)$ be a subtorus. Then $\fg/\fc$ is generically toral, and the canonical projection
$\pi : \fg \lra \fg/\fc$ induces an isomorphism $C(\fg)/\fc \stackrel{\sim}{\lra} C(\fg/\fc)$. \end{Lem} 

\begin{proof} Let $\ft \subseteq \fg$ be a torus of dimension $\mu(\fg)$. Then $\ft' := \pi(\ft)$ is a torus of $\fg':= \fg/\fc$.

Let $\pi(x)$ be an element of the centralizer of $\ft'$ in $\fg'$. Then we have $[x,t] \in \fc \subseteq \ft$ for every $t \in \ft$. Since $\ft$ is a Cartan subalgebra of
$\fg$, it follows that $x \in \ft$, so that $\pi(x) \in \ft'$. As a result, $\ft'$ is a self-centralizing torus of $\fg'$, so that $\fg'$ is generically toral.

Let 
\[ \fg := \ft\!\oplus\!\bigoplus_{\alpha \in R_\ft}\fg_\alpha\]
be the root space decomposition of $\fg$ relative to $\ft$, so that $C(\fg)=\bigcap_{\alpha \in R_\ft}\ker\alpha$. 

The surjection $\pi$ clearly induces an injection $C(\fg)/\fc \hookrightarrow C(\fg')$. Let $\pi(z) \in C(\fg')$ be a central element. Given $x \in \fg$, we have
$[z,x] \in \fc$. Let $t \in \ft$. As $\ft$ is a torus, there is $s \in \ft$ such that $t=s^{[p]}$, whence
\[ [t,z] = [s^{[p]},z] \in \ad(s)^{p-1}(\fc)=(0).\]
Hence $z \in \ft$ and 
\[\alpha(z)x = [z,x] \in \fg_\alpha\cap\fc = (0)\]
for all $x \in \fg_\alpha$ and $\alpha \in R_\ft$. Thus, $z \in \bigcap_{\alpha \in R_\ft} \ker\alpha = C(\fg)$. As a result, the above injection is also surjective.
\end{proof}

\bigskip

\begin{Lem} \label{Gt2} Suppose that $\fg$ is generically toral. If $\ft \in \mstor(\fg)$, then 
\[\fg_\alpha^{[p]} \subseteq \ker\alpha\] 
for every $\alpha \in R_\ft$. \end{Lem}

\begin{proof} Note that $C(\fg) \subseteq C_\fg(\ft)=\ft$. Given $x \in \fg_\alpha\!\smallsetminus\!\{0\}$, we have
\[ (\ad x^{[p]})(\ft) = (\ad x)^p(\ft) \subseteq \fg_{p\alpha} = \ft,\] 
whence $x^{[p]} \in {\rm Nor}_\fg(\ft)$, the normalizer of $\ft$ in $\fg$. Since the Cartan subalgebra $\ft$ is self-normalizing, it follows that
$x^{[p]} \in \ft$. Moreover,
\[ 0 = [x^{[p]},x] = \alpha(x^{[p]})x,\]
so that $x^{[p]} \in \ker \alpha$. \end{proof}

\bigskip

\subsection{Centralizers and $2$-saturation}
A restricted Lie algebra $(\fe,[p])$ is called \textit{elementary abelian}, provided $[\fe,\fe]=\{0\}=V(\fe)$. Recall that the subset
\[ \EE(r,\fg):=\{\fe \in \Gr_r(\fg) \ ; \ \fe \subseteq \fg \ \text{is elementary abelian}\}\]
of the Grassmannian $\Gr_r(\fg)$ is closed and hence a projective variety. This implies that
\[ V_{\EE(r,\fg)} := \bigcup_{\fe \in \EE(r,\fg)} \fe\]
is a conical, closed subset of $V(\fg)$. We say that $\fg$ is \textit{$r$-saturated}, provided $V(\fg)=V_{\EE(r,\fg)}$.  

Let $\rk(\fg)$ be the minimal dimension of all Cartan subalgebras of $\fg$, the so-called \textit{rank} of $\fg$. We denote by $T(\fg)$ the 
\textit{toral radical} of $\fg$, that is, the unique maximal toral ideal of $\fg$. Note that $T(\fg)$ is contained in the center $C(\fg)$. For $x \in \fg$,
$(kx)_p$ denotes the $p$-subalgebra of $\fg$ that is generated by $x$. We have $(kx)_p = \sum_{i\ge 0}kx^{[p]^i}$, so that $(kx)_p$ is contained
in the centralizer $C_\fg(x)$ of $x$ in $\fg$. 

\bigskip

\begin{Lem} \label{CenSat1} Let $(\fg,[p])$ be a restricted Lie algebra. Then the following statements hold:
\begin{enumerate}
\item If $\fh \subseteq \fg$ is a $p$-subalgebra such that $\dim V(\fh) = 1$, then
\[ \dim_k\fh \le \mu(\fh)\!+\!\rk(\fg)\]
\item If $x \in \fg$ is such that 
\begin{enumerate}
\item $x^{[p]} \in T(\fg)$, and
\item $\dim V(C_\fg(x)) \le 1$, \end{enumerate}
then $\dim_k\fg \le p(\mu(C_\fg(x))\!+\!\rk(\fg))$. \end{enumerate} \end{Lem}

\begin{proof} (1) Since $\dim V(\fh)=1$, there exists a torus $\ft \subseteq \fh$ and a $p$-nilpotent element $y \in \fh$ such that
\[ \fh = \ft\!\oplus\!(ky)_p,\]
cf.\ for instance \cite[(4.3)]{Fa95}. The dimension of the first summand is bounded by $\mu(\fh)$, while \cite[(8.6(3))]{Fa04} implies $\dim_k(ky)_p 
\le \rk(\fg)$.

(2) Since $x^{[p]} \in T(\fg) \subseteq C(\fg)$, we have $(\ad x)^p=0$, so that $\fg$ affords a decomposition
\[ \fg = \bigoplus_{i=1}^p a_i[i]\]
into $(\ad x)$-cyclic subspaces, where $\dim_k[i]=i$. It follows that $\dim_kC_\fg(x)=\sum_{i=1}^pa_i$, while
\[ \dim_k\fg = \sum_{i=1}^pia_i \le p(\sum_{i=1}^pa_i) = p\dim_kC_\fg(x).\]
Using (1) we arrive at 
\[ \dim_k\fg \le p(\mu(C_\fg(x))\!+\!\rk(\fg)),\]
as desired. \end{proof}  

\bigskip

\begin{Remark} Lemma \ref{CenSat1} provides no information for nilpotent Lie algebra, as we have $\dim_k\fg = \rk(\fg)$ in that case. This is one
reason for confining our attention to generically toral Lie algebras. \end{Remark} 

\bigskip

\begin{Lem} \label{CenSat2} Let $(\fg,[p])$ be generically toral. Suppose there is $x \in V(\fg)\smallsetminus\!\{0\}$ such that $\dim V(C_\fg(x))=1$. 
Then we have
\[ \dim_k\fg \le p(2\mu(\fg)\!-\!1\!-\!\dim_kC(\fg)).\]
\end{Lem}

\begin{proof} In view of Lemma \ref{Gt1}, the algebra $\fg':= \fg/C(\fg)$ is generically toral and \cite[(3.5),(3.6)]{Fa04} implies $\rk(\fg')=\mu(\fg')$. As before, we write
\[ C_\fg(x) = \ft\!\oplus\!(ky)_p,\]
for some torus $\ft \subseteq \fg$ and some $p$-nilpotent element $y \in \fg$.  Since $C(\fg)$ is a torus, the canonical projection
$\pi : \fg \lra \fg'$ provides an isomorphism $\pi : (ky)_p \lra (k\pi(y))_p$, so that \cite[(8.3)]{Fa04} in conjunction with \cite[(3.3)]{FV03} yields $\dim_k(ky)_p
\le \rk(\fg')=\mu(\fg')=\mu(\fg)\!-\!\dim_kC(\fg)$. The assumption $\mu(C_\fg(x))=\mu(\fg)$ implies that there is $\ft \in \mstor(\fg)$ with $x \in C_\fg(\ft)$. As $\ft$ is 
self-centralizing, we have reached a contradiction. Consequently, $\mu(C_\fg(x))\le \mu(\fg)\!-\!1$. The assertion now follows from Lemma \ref{CenSat1}. 
\end{proof}  

\bigskip

\begin{Cor} \label{CenSat3} Let $(\fg,[p])$ be a restricted Lie algebra. 
\begin{enumerate} 
\item If
\[ \dim_k\fg > p(\mu(\fg)\!+\!\rk(\fg)),\]
then $\fg$ is $2$-saturated. 
\item If $\fg$ is generically toral and 
\[ \dim_k\fg > p(2\mu(\fg)\!-\!1\!-\!\dim_kC(\fg)),\]
then $\fg$ is $2$-saturated. \end{enumerate} \end{Cor}

\begin{proof} Let $x \in V(\fg)\!\smallsetminus\!\{0\}$. In view of Lemma \ref{CenSat1} and Lemma \ref{CenSat2}, we have $\dim V(C_\fg(x)) \ge 2$. Hence 
there is $y \in V(C_\fg(x))\!\smallsetminus\!\{0\}$ such that $kx \ne ky$. Consequently, $kx\!\oplus\!ky \in \EE(2,\fg)$ and $x \in V_{\EE(2,\fg)}$. \end{proof}

\bigskip

\begin{Examples} We shall show that the restricted Lie algebras of Cartan type are $2$-saturated.
\begin{enumerate}
\item Let $W(n)$ be the Jacobson-Witt algebra. Then $W(n)$ is generically toral and $n=\mu(\fg)=\rk(\fg)$, while
\[ \dim_kW(n)=np^n > p(2n\!-\!1)\]
unless $n\!=\!1$. It follows that $V(W(n))=V_{\EE(2,W(n))}$ for $n\!\ge\! 2$. For $n\!=\!1$, the result follows from the fact that $\dim_kC_\fg(x)<p$ and 
$\mu(C_\fg(x))=0$ for every $x \in V(W(1))$. 

\item Let $S(n)$ be the special Lie algebra, where $n\ge 3$. Then we have $\dim_k S(n) = (n\!-\!1)(p^n\!-\!1)$ and Demushkin's work \cite[Thm.2,Cor.2]{De1}
yields $\rk(S(n)) = (n\!-\!1)(p\!-\!1)$. Thus,
\[ \dim_k S(n) = (\sum_{i=0}^{n-1}p^i)\rk(S(n)) > 2p\rk(S(n)),\]
and Corollary \ref{CenSat3} implies $V(S(n))=V_{\EE(2,S(n))}$.

\item Consider the Hamiltonian algebra $H(2r)$. Then we have $\dim_kH(2r)=p^{2r}\!-\!2$, while $\rk(H(2r))=p^r\!-\!2$, cf.\ \cite[Cor.]{De2}. Consequently,
\[ 2p\rk(H(2r)) = 2p^{r+1}\!-\!4p < p^{2r}\!-\!2,\]
unless $r\!=\!1$. In that case, we have $\mu(H(2))=1$, and
\[ p(\rk(H(2))\!+\!\mu(H(2))) = p(p\!-\!2\!+\!1)=p^2\!-\!p <p^2\!-\!2 = \dim_kH(2).\]
Thus, Corollary \ref{CenSat3} shows that $H(2r)$ is $2$-saturated. 

\item Consider the contact algebra $K(2r\!+\!1)$. Demushkin's work yields $\mu(K(2r\!+\!1)) = r\!+\!1$. Thanks to \cite[(3.3)]{Ch16a}, we have $\rk(K(2r\!+\!1)) \in \{p^r,p^r\!-\!1\}$. 
Consequently,
\begin{eqnarray*}
\dim_kK(2r\!+\!1) & \ge & p^{2r+1}\!-\!1 = (p\!-\!1)(\sum_{i=0}^{2r}p^i) > p(\sum_{i=0}^{2r-1}p^i)\\ 
& \ge & p(1\!+p(2r\!-\!2)+p^r) \ge p(r\!+\!1\!+\!p^r)\\
& \ge & p(\mu(K(2r\!+\!1))\!+\rk(K(2r\!+\!1)))
\end{eqnarray*}
for $r \ge 2$. For $r=1$, we have
\[ \dim_kK(3) \ge p^3\!-\!1 > p(2\!+\!p) \ge p(\mu(K(3))\!+\!\rk(K(3))).\]
The desired result now follows from Corollary \ref{CenSat3}. \end{enumerate} \end{Examples}

\bigskip

\section{Central Extensions}
Let $(\fg,[p])$ be a restricted Lie algebra. Given a $\fg$-module $M$, we denote by $\HH^n(\fg,M)$ the $n$-th Chevalley-Eilenberg cohomology group of $\fg$
with coefficients in $M$. Recall that $\HH^1(\fg,M)$ is isomorphic to the quotient $\Der_k(\fg,M)/{\rm Inn}(\fg,M)$ of derivations by inner derivations, while 
$\HH^2(\fg,M)$ describes the equivalence classes of abelian extensions of $\fg$ by $M$, cf.\ \cite[VII]{HS} for more details. A special case is given by a 
central extension
\[ (0) \lra \fc \lra \fg \lra \fg/\fc \lra (0),\]
where $\fc \subseteq C(\fg)$ is a torus. If $\HH^2(\fg/\fc,k)=(0)$, then such an extension splits in the category of ordinary Lie algebras, but not necessarily
in the category of restricted Lie algebras. However, one does have good control of the $p$-map in this situation.

\bigskip  

\begin{Lemma} \label{CE1} Let $(\fg,[p])$ be a restricted Lie algebra with toral center.
\begin{enumerate}
\item The canonical projection $\pi : \fg \lra \fg/C(\fg)$ induces an injective, closed morphism
\[ \pi_\ast : \EE(2,\fg) \lra \EE(2,\fg/C(\fg)) \ \ ; \ \ \fe \mapsto \pi(\fe).\]
\item If $\HH^2(\fg/C(\fg),k)=(0)$, then $\pi_\ast$ is bijective. \end{enumerate}
\end{Lemma}

\begin{proof} (1) By assumption, the center $C(\fg)$ is a torus. This readily implies that
\[ \pi_\ast : \EE(2,\fg) \lra \EE(2,\fg/C(\fg))\]
is a morphism of projective varieties. In particular, the map $\pi_\ast$ is closed. 

Let $\fe,\ff \in \EE(2,\fg)$ be such that $\pi(\fe)=\pi(\ff)$. Then $\fe\!\oplus\!C(\fg)= \ff\!\oplus\!C(\fg)$, so that $\fe = V(\fe\!\oplus\!C(\fg)) = V(\ff\!\oplus\!C(\fg)) = \ff$.  
Hence $\pi_\ast$ is injective.

(2) Since $\HH^2(\fg/C(\fg),k)=(0)$, the exact sequence
\[ (0) \lra C(\fg) \lra \fg \lra \fg/C(\fg) \lra (0)\]
splits as a sequence of ordinary Lie algebras. Hence there is a subalgebra $\fh \subseteq \fg$ such that
\[ \fg = \fh\!\oplus\!C(\fg).\]
As $\fh \cong \fg/C(\fg)$, there is a $p$-map on $\fh$ such that $\pi|_\fh$ is an isomorphism of restricted Lie algebras. General theory (cf.\ \cite[(II.2.1)]{SF}) provides a 
$p$-semilinear map $\lambda : \fh \lra C(\fg)$ such that
\[ (h,c)^{[p]} = (h^{[p]},\lambda(h)\!+\!c^{[p]})\]
for all $(h,c) \in \fg$. Let $\fe \in \EE(2,\fg/C(\fg))$. Then there exist $x,y \in V(\fh)$ such that $\fe = k\pi(x,0)\!\oplus\!k\pi(y,0)$. Since $C(\fg)$ is a torus, 
we can find $v,w \in C(\fg)$ such that $v^{[p]}=-\lambda(x)$ and $w^{[p]}=-\lambda(y)$. It follows that $(x,v), (y,w) \in V(\fg)$, while $[(x,v),(y,w)] = ([x,y],0) 
= (0,0)$. As a result, $\ff := k(x,v)\!\oplus\!k(y,w) \in \EE(2,\fg)$ and $\pi_\ast(\ff) = \fe$. Consequently, the morphism $\pi_\ast$ is surjective. \end{proof}

\bigskip

\begin{Example} Let $\fh:= kx\!\oplus\!ky\!\oplus\!kz$ be the Heisenberg Lie algebra with toral center. By definition, its Lie bracket and $p$-map are given by
\[ [x,y]=z \ , \ [x,z]=0=[y,z] \ \ \text{and} \ \ x^{[p]}=0=y^{[p]} \ , \ z^{[p]}=z,\]
respectively. Then we have $\EE(2,\fh)=\emptyset$, while $\EE(2,\fh/C(\fh))=\{\fh/C(\fh)\}$. \end{Example}

\bigskip
\noindent
Let $(\fg,[p])$ be a generically toral restricted Lie algebra. We say that $\fg$ is {\it freely generated}, if there exists a maximal torus $\ft \in \mstor(\fg)$ and 
a subset $S_\ft \subseteq R_\ft$ such that
\begin{enumerate}
\item[(a)] $[\fg_\alpha,\fg_{-\alpha}] \ne (0)$ for all $\alpha \in S_\ft$, and 
\item[(b)] $\ft = \bigoplus_{\alpha \in S_\ft} [\fg_\alpha,\fg_{-\alpha}]$, and 
\item[(c)] the Lie algebra $\fg$ is generated by $\bigcup_{\alpha \in S_\ft}\fg_{\alpha}$.
\end{enumerate} 
In this case, we call $\ft$ a \textit{distinguished torus}.

\bigskip

\begin{Lemma} \label{CE2} Suppose that $\fg$ is freely generated. If $\rho(\fg)=1$, then $\HH^2(\fg,k)=(0)$. \end{Lemma}

\begin{proof} Let $\ft \in \mstor(\fg)$ be distinguished with root space decomposition
\[ \fg = \ft \!\oplus\!\bigoplus_{\alpha \in R_\ft}\fg_\alpha.\]
We let $S_\ft \subseteq R_\ft$ be a subset such that the defining conditions (a), (b) and (c) hold. Since $\rho(\fg)=1$, we write $\fg_\alpha = kx_{\alpha}$ for every $\alpha \in R_\ft$.

Let $\varphi : \fg \lra \fg^\ast$ be a derivation of degree $0$, that is, $\varphi(\fg_\alpha) \subseteq (\fg^\ast)_\alpha$ for all $\alpha \in R_\ft$. Owing to (a) and
(b), there is a linear form 
$f \in (\fg^\ast)_0$ such that
\[ f([x_\alpha,x_{-\alpha}]) = -\varphi(x_\alpha)(x_{-\alpha}) \ \ \ \ \ \ \forall \ \alpha \in S_\ft.\]
We let 
\[ \ad f : \fg \lra \fg^\ast \ \ ; \ \ x \mapsto x\dact f\]
be the inner derivation effected by $f$ and consider the derivation $\psi := \varphi\!-\!\ad f$. Given $\alpha \in S_\ft$, we have $\psi(x_\alpha) \in (\fg^\ast)_\alpha$ 
and 
\[ \psi(x_\alpha)(x_{-\alpha}) = \varphi (x_\alpha)(x_{-\alpha}) \!-\! (x_\alpha\dact f)(x_{-\alpha}) = \varphi (x_\alpha)(x_{-\alpha}) \!+\! f([x_\alpha,x_{-\alpha})]
=0.\]
Since $\psi$ has degree $0$, this implies that $\psi(\sum_{\alpha \in S_\ft}\fg_\alpha)=(0)$.  As $\fg$ is generated by $\sum_{\alpha \in S_\ft}\fg_\alpha$, it follows 
that $\psi =0$, so that $\varphi = \ad f$ is an inner derivation. Thanks to \cite[(1.2)]{Fa88}, this forces all derivations to be inner, whence $\HH^1(\fg,\fg^\ast) =(0)$. 
The assertion now follows from the canonical inclusion $\HH^2(\fg,k) \hookrightarrow \HH^1(\fg,\fg^\ast)$, cf.\ \cite[(1.3)]{Fa86}. \end{proof}  

\bigskip

\begin{Remark} Condition (a) is essential for the validity of Lemma \ref{CE2}. The Witt algebra $W(1):=\bigoplus_{i=-1}^{p-2}ke_i$ is generically toral with maximal 
torus $ke_0$ and root spaces $W(1)_i:=ke_i$. It is generated by $W(1)_{-1}\!\oplus \!W(1)_2$, and we have $ke_0=[W(1)_{-1},W(1)_1]$, while 
$[W(1)_2,W(1)_{-2}]=(0)$. For $p\ge 5$, the space $\HH^2(W(1),k)$ is known to be one-dimensional (cf.\ \cite[(5.1)]{Bl}). \end{Remark}

\bigskip

\section{Lie algebras with enough roots}\label{S:ER}
Let $(\fg,[p])$ be a restricted Lie algebra. Recall that
\[ \msrke(\fg) = \max\{r \in \NN_0 \ ; \ \EE(r,\fg) \ne \emptyset\}\]
is the $p$-rank of $\fg$. As noted before, a result by Chwe \cite{Ch65} asserts that $\fg$ is a torus if and only if $\msrke(\fg)=0$.

In this section we show that, in the context of Lie algebras with $\msrke(\fg)=1$, our assumption concerning the existence of self-centralizing maximal tori is fulfilled in most cases of interest. 

\bigskip

\begin{Proposition} \label{LR1} Let $(\fg,[p])$ be a restricted Lie algebra such that $\msrke(\fg)=1$. If $r(\fg) \ge 2$, then $\fg$ is generically toral.
\end{Proposition}

\begin{proof} Let $\ft \in \mstor(\fg)$ be a torus of dimension $\mu(\fg)$, $\fh := C_\fg(\ft)$ be the corresponding Cartan subalgebra of dimension $\rk(\fg)$
with root space decomposition
\[ \fg = \fh\!\oplus\!\bigoplus_{\alpha \in R_\ft}\fg_\alpha.\]
Suppose there is $x \in V(\fh)\!\smallsetminus\!\{0\}$. Since $[\ft,\fh]=(0)$, we have $\ft \subseteq C_\fg(x)$, and there results a weight space decomposition
\[ C_\fg(x) = C_\fg(x)_0\!\oplus\!\bigoplus_{\alpha \in R_\ft}C_\fg(x)_{\alpha}\]
of $C_\fg(x)$ relative to $\ft$. As $x$ acts nilpotently on each root space $\fg_\alpha$, we have $C_\fg(x)_\alpha \ne (0)$ for all $\alpha \in R_\ft$. 
Moreover, $\ft$ is a maximal torus of $C_\fg(x)$, so that $C_\fg(x)_0 = \fh \cap C_\fg(x)$ is a Cartan subalgebra of $C_\fg(x)$. 

In view of $\msrke(\fg) = 1$, we have $V(C_\fg(x))=kx$, and \cite[(3.2),(4.3)]{Fa95} provides a toral element $t \in C_\fg(x)$ and a $p$-nilpotent element 
$y \in C_\fg(x)$ such that
\[ C_\fg(x) = kt\!\ltimes\!I,\]   
where $I:=T(C_\fg(x))\!\oplus\!(ky)_p$ is an abelian $p$-ideal of $C_\fg(x)$. Consequently, $I^{[p]}=T(C(\fg))\!\oplus\!(ky^{[p]})_p$ lies in the center $Z(x)$ of 
$C_\fg(x)$. Since $\fh\cap C_\fg(x)$ is a Cartan subalgebra of $C_\fg(x)$, it follows that 
\[ T(C_\fg(x))\!\oplus\!(ky^{[p]})_p \subseteq Z(x) \subseteq C_\fg(x)\cap\fh.\] 
As $\ft$ is the set of semisimple elements of $\fh$, it follows that $T(C_\fg(x)) \subseteq \ft$. Observing that $(\ft\!+\!I)/I$ is a maximal torus of $C_\fg(x)/I$, while 
$\ft\cap(ky)_p = (0)$, we conclude that $C_\fg(x)=\ft\!\oplus\!(ky)_p$. Moreover, $\ft\!\oplus (ky^{[p]})_p\subseteq C_\fg(x)_0$, so that 
$\dim_kC_\fg(x)/C_\fg(x)_0 \le 1$. If $C_\fg(x)=C_\fg(x)_0$, then $R_\ft=\emptyset$ and $r(\fg)=0$, a contradiction. Alternatively, $\ft$ acts on the 
one-dimensional space $C_\fg(x)/C_\fg(x)_0$, so that we find $\alpha \in R_\ft$ and $x_\alpha \in C_\fg(x)_\alpha\!\smallsetminus\!\{0\}$ such that
\[ C_\fg(x) = C_\fg(x)_0\!\oplus\!kx_\alpha.\]
In view of the above, we thus have $\fg = \fh\!\oplus\!\fg_\alpha$, so that $r(\fg)=1$, a contradiction.

Hence $V(\fh)=\{0\}$, implying that $\fh$ is a torus.  As a result, the restricted Lie algebra $\fg$ is generically toral. \end{proof}

\bigskip

\begin{Remark} The foregoing result fails for Lie algebras with $r(\fg)=1$: Let $\fg := kt\!\oplus\!kx\!\oplus\!ky$, where
\[ [y,\fg] = (0) \ , \ [t,x] = x \ \ ; \ \ t^{[p]} = t \ , \ x^{[p]}=y \ , \ y^{[p]}=0.\]
Then $\fh := kt\!\oplus\!ky$ is a Cartan subalgebra of $\fg$, and $V(\fg)=V(\fh)=ky$. Hence $\mu(\fg)=1$, $\fg$ is not generically toral, and
$\msrke(\fg)=1$.  \end{Remark} 

\bigskip

\begin{Corollary} \label{LR2} Let $(\fg,[p])$ be a restricted Lie algebra such that $\msrke(\fg)=1$. If $\fg$ is centerless or not solvable, then $\fg$ is 
generically toral. \end{Corollary}

\begin{proof} Let $\ft \in \mstor(\fg)$, 
\[ \fg = C_\fg(\ft)\!\oplus\!\bigoplus_{\alpha \in R_\ft}\fg_\alpha\]
be the corresponding root space decomposition. 

If $r(\fg)=0$, then $\fg$ is nilpotent, hence not centerless and solvable, a contradiction. 

If $r(\fg)=1$, then $R_\ft=\{\alpha\}$, and $\fg_\alpha$ is an abelian ideal. This implies that $\fg$ is solvable. By assumption, $\fg$ is centerless, so that $\fg_\alpha^{[p]} = \{0\}$. 
Let $x \in V(C_\fg(\ft))$. Since $x$ acts nilpotently on $\fg_\alpha$, there is $x_\alpha \in \fg_{\alpha}\!\smallsetminus\!\{0\}$ such that $[x,x_\alpha] = 0$.
As $\msrke(\fg)=1$, we have $x=0$, so that $C_\fg(\ft)$ is a torus. The assertion now follows from Proposition \ref{LR1}. \end{proof} 

\bigskip

\section{Lie algebras with $\msrke(\fg)=1$}\label{S:Rk1}
Our goal is to classify those generically toral algebras that do not possess a two-dimensional elementary abelian Lie algebra. These are the Lie algebras with 
$\msrke(\fg)\le 1$. By way of illustration, we first consider algebraic Lie algebras.

\bigskip

\subsection{Algebraic Lie algebras}
Throughout this subsection, we let $G$ be an algebraic group with Lie algebra $\fg=\Lie(G)$. 

\bigskip

\begin{Thm} \label{AL1} The Lie algebra $\fg=\Lie(G)$ has rank $\msrke(\fg) \le 1$ if and only if there exists a torus $\ft \subseteq \fg$ such that
\begin{enumerate}
\item[(a)] $\fg \cong \ft\!\oplus\!\fsl(2)$, or
\item[(b)] $\fg \cong \ft\!\ltimes\!(kx)_p$, for some $p$-nilpotent element $x \in \fg$. \end{enumerate}\end{Thm}

\begin{proof} Let $U\subseteq G$ be the unipotent radical of $G$ and put $\fu:=\Lie(U)$. Then $G$ acts on $\fu$ via the adjoint representation, so that
$C(\fu)$ and $V(C(\fu))$ are $G$-stable subsets of $\fg$. 

Suppose that $U\ne e_k$, so that $\fu\ne (0)$. Then $C(\fu)\ne(0)$ is $p$-unipotent, and \cite{Ch65} yields $V(C(\fu))\ne\{0\}$. Since $\msrke(C(\fu))\le
\msrke(\fg) \le 1$, it follows that $V(C(\fu))=kx$ is one-dimensional. As a result, $kx$ is an ideal of $\fg$. 

Let $y \in V(\fg)$. Since $y$ acts nilpotently on $kx$, we obtain $[x,y]=0$. Hence $kx\!+\!ky$ is elementary abelian, so that $y \in kx$. Consequently,
$V(\fg)=kx$, and $\fg$ is of type (b).

Alternatively, the group $G$ is reductive. Let $T\subseteq G$ be a maximal torus with Lie algebra $\ft$,
\[ \fg = \ft\!\oplus\!\bigoplus_{\alpha \in \Phi}\fg_\alpha\]
be the root space decomposition of $\fg$ relative to $T$. If $\Phi=\emptyset$, then $\fg=\ft$ is of type (b). Alternatively, let $\Delta \subseteq \Phi^+\subseteq \Phi$ be subsystems of simple 
and positive roots, respectively. General theory tells us that $\bigcup_{\alpha \in \Phi}\fg_\alpha \subseteq V(\fg)$. Let $\alpha_0 \in \Phi$ be a root of maximal height.
If there is a positive root $\alpha_1 \in \Phi^+\!\smallsetminus\!\{\alpha_0\}$, then $[\fg_{\alpha_1},\fg_{\alpha_0}]=(0)$. As this contradicts $\msrke(\fg)
\le 1$, we conclude that $|\Phi^+|=1$. Consequently, $\Phi =\{\alpha_0,-\alpha_0\}$, and $G$ is of type $A_1$. This implies that $\fg$ is of type (a).

Finally, Lie algebras of type (a) or (b) are easily seen to have $p$-rank $\msrke(\fg)\le 1$. \end{proof} 

\bigskip

\subsection{Centralizers and roots}

\begin{Prop} \label{CR1} Let $\fg$ be generically toral and such that $\msrke(\fg)=1$. Given $\ft \in \mstor(\fg)$, the following statements
hold:
\begin{enumerate}
\item There is a root space decomposition $\fg = \ft\!\oplus\!\bigoplus_{\alpha \in R_\ft}\fg_\alpha$.
\item We have $C_\fg(x_\alpha)=(\ker\alpha)\!\oplus\!kx_\alpha$ for every $\alpha \in R_\ft$ and $x_\alpha \in \fg_\alpha\!\smallsetminus\!\{0\}$.
\item If $\dim_k\fg_\alpha=\rho(\fg)$, then $\dim_k\fg_\beta=1$ for all $\beta \in R_\ft\smallsetminus\!\{\alpha,-\alpha\}$.
\item If $\dim_k\fg_\alpha=\rho(\fg)$ and $R_\ft\!\smallsetminus\!\{\alpha,-\alpha,-2\alpha\} \ne \emptyset$, then $\rho(\fg)=1$.\end{enumerate}
\end{Prop}

\begin{proof} (1) Owing to \cite[(3.8)]{Fa04}, we have $C_\fg(\ft)=\ft$.

(2) Let $\alpha \in R_\ft$, $x_\alpha \in \fg_\alpha\!\smallsetminus\!\{0\}$. Then the centralizer $C_\fg(x_\alpha)$ is a $\ft$-stable $p$-subalgebra. Observing 
$x_\alpha \ne 0$, we obtain a weight space decomposition 
\[ C_\fg(x_\alpha) = (\ker \alpha)\!\oplus\!\bigoplus_{\beta \in R_\ft}C_\fg(x_\alpha)\cap\fg_\beta\]
relative to $\ft$. Since $\ker \alpha \subseteq \ft$ is a subtorus, Lemma \ref{Gt2} ensures the existence of $t_\alpha \in \ker\alpha$ such that $x_\alpha^{[p]} = 
t_\alpha^{[p]}$. This implies that $x_\alpha\!-\!t_\alpha \in V(\fg)\cap C_\fg(x_\alpha)$. Note that 
\[ f_\alpha : C_\fg(x_\alpha) \lra C_\fg(x_\alpha) \ \ ; \ \ y \mapsto [t_\alpha, y] \] 
is a semisimple linear transformation such that $f_\alpha^p=0$. Consequently, $f_\alpha=0$, so that $x_\alpha\!-\!t_\alpha$ belongs to the center $Z(x_\alpha)$ 
of $C_\fg(x_\alpha)$. Since $\msrke(\fg)=1$, it follows that $V(\fg)\cap C_\fg(x_\alpha)=k(x_\alpha\!-\!t_\alpha)$. 

Let $\ft_\alpha$ be the unique maximal toral ideal of $C_\fg(x_\alpha)$. Owing to \cite[(4.3)]{Fa95}, there exist a $p$-nilpotent element $y \in C_\fg(x_\alpha)$, 
and a toral element $t \in C_\fg(x_\alpha)$ such that
\[ C_\fg(x_\alpha) = kt\!\ltimes\!(\ft_\alpha\!\oplus (ky)_p),\]
where $\fn_\alpha:= \!\ft_\alpha\!\oplus (ky)_p$ is an abelian $p$-ideal of $C_\fg(x_\alpha)$. 

Since $V(C_\fg(x_\alpha))=k(x_\alpha\!-\!t_\alpha)$ there are $\lambda \in k^\times$ and $n \ge 0$ such that $y^{[p]^n}=\lambda(x_\alpha\!-\!t_\alpha)$.
In view of $[\ft,y^{[p]}] \subseteq (\ad y)^{p-1}(C_\fg(x_\alpha)) \subseteq (\ad y)^2(C_\fg(x_\alpha))= (0)$, while $[\ft,x_\alpha\!-\!t_\alpha] = kx_\alpha$, it follows 
that $n=0$. We therefore have $C_\fg(x_\alpha)=kt\!\oplus\ft_\alpha\!\oplus k(x_\alpha\!-\!t_\alpha)$. 

Since $\ft$ acts on $C_\fg(x_\alpha)$ by derivations, the center $Z(x_\alpha)$ is $\ft$-stable. As $\ft_\alpha$ is a toral ideal in $Z(x_\alpha)$, it follows that 
$(\ad s)(\ft_\alpha)=(0)$ for every $s \in \ft$. Consequently, 
$\ft_\alpha \subseteq C_\fg(x_\alpha)_0 = \ker \alpha$. If $\ft_\alpha \subsetneq \ker\alpha$, then we have $C_\fg(x_\alpha) = (\ker\alpha)\!\oplus\! 
k(x_\alpha\!-\!t_\alpha)$ for dimension reasons. Thus, $C_\fg(x_\alpha)= (\ker\alpha)\!\oplus\!kx_\alpha$.

Alternatively, we have $\ft_\alpha = \ker\alpha$. But then $[t,\ker\alpha] =(0)$, whence $[t,x_\alpha\!-\!t_\alpha] = 0$ as $t \in C_\fg(x_\alpha)$. Hence 
$[t,C_\fg(x_\alpha)] = (0)$, so that $t=0$ and $C_\fg(x_\alpha)=(\ker\alpha)\!\oplus\!kx_\alpha$. 

(3) Suppose there is $\beta \in R_\ft\!\smallsetminus\!\{\alpha,-\alpha\}$ such that $\dim_k\fg_\beta \ge 2$. Then Corollary \ref{RSD2} provides $x \in 
\fg_{\alpha}\!\smallsetminus\!\{0\}$ and  $y \in \fg_{\beta}\!\smallsetminus\!\{0\}$ such that $[x,y]=0$. Thus, $y \in C_\fg(x)$, which contradicts (2). 

(4) Let $\beta \ne \alpha, -\alpha, -2\alpha$ be a root, $x_\beta \in \fg_\beta\!\smallsetminus\!\{0\}$. In view of (2), the map
\[ f_\beta : \fg_\alpha \lra \fg_{\alpha+\beta} \ \ ; \ \ v \mapsto [x_\beta,v]\]
is injective. Since $\alpha\!+\!\beta \in R_\ft\!\smallsetminus\!\{\alpha,-\alpha\}$, part (3) implies $\dim_k\fg_\alpha \le \dim_k\fg_{\alpha+\beta} \le 1$.  \end{proof}

\bigskip
\noindent
Let $\ft \in \mstor(\fg)$ be maximal torus. Then $R_\ft \subseteq \ft^\ast$, and we put $r_\ft:= \dim_k\langle R_\ft \rangle$. If $\ft$ is self-centralizing, then
$\langle R_\ft \rangle^\perp = C(\fg)$, so that
\[ r_\ft = \dim_k\ft\!-\!\dim_kC(\fg) = \mu(\fg)\!-\!\dim_kC(\fg).\]
Given $\ft \in \mstor(\fg)$, we denote by $\ft_p := \{t \in \ft \ ; \ t^{[p]}=t\}$ the $\FF_p$-subspace of its toral elements. General theory implies that 
$\dim_{\FF_p}\ft_p=\mu(\fg)$. Since $\alpha(t^{[p]})=\alpha(t)^p$ for all $\alpha \in R_\ft$ and $t \in \ft$, it follows that every $\alpha \in R_\ft$ is
uniquely determined by its restriction $\alpha|_{\ft_p} \in \Hom_{\FF_p}(\ft_p,\FF_p)$. Moreover,
\[ r_\ft = \dim_{\FF_p} \FF_p\{\alpha|_{\ft_p} \ ; \ \alpha \in R_\ft\}\]
is the dimension of the $\FF_p$-span of $R_\ft$. By abuse of notation, we will henceforth consider $R_\ft$ a subset of $\Hom_{\FF_p}(\ft_p,\FF_p)$ whenever
this is convenient.
  
Let $k_{-1}$ be the one-dimensional restricted $\fb_{\fsl(2)}$-module, on which the given toral element $t \in \fb_{\fsl(2)}$ acts via $-1$. Thus, the 
semidirect product 
\[\fb_{\fsl(2)}^{-1}:= \fb_{\fsl(2)}\!\ltimes\!k_{-1} = kt\!\oplus\!kx\!\oplus\!ky\] 
has the following structure of a restricted Lie algebra:
\[ [t,x] = x \ , \ [t,y] = -y \ , \ [x,y] = 0 \ \ ; \ \ t^{[p]}=t \ , \ x^{[p]}=0=y^{[p]}.\]
Note that $\fb_{\fsl(2)}^{-1}$ is centerless and generically toral such that $\mu(\fb_{\fsl(2)}^{-1})=1=\rho(\fb_{\fsl(2)}^{-1})$. 
Moreover, $\EE(2,\fb_{\fsl(2)}^{-1})=\{kx\!\oplus\!ky\}$. 

\bigskip

\begin{Lem} \label{CR2} Let $\fg$ be a generically toral restricted Lie algebra such that $\msrke(\fg)=1=\rho(\fg)$. If $\mu(\fg)=1\!+\!\dim_kC(\fg)$, then 
\[\fg/C(\fg) \cong \fb_{\fsl(2)},\, \fb_{\fsl(2)}^{-1},\, \fsl(2).\] 
Moreover, if $\fg/C(\fg) \cong \fb_{\fsl(2)}^{-1}$, then $C(\fg)\ne (0)$.\end{Lem}

\begin{proof} Let $\ft \in \mstor(\fg)$. By assumption, we have $r_\ft=1$, so that there is $\alpha \in R_\ft$ with $R_\ft \subseteq \FF_p\alpha$.
In view of Lemma \ref{Gt2}, we have $\fg_{i\alpha}^{[p]} \subseteq \ker\alpha = C(\fg)$ for all $i\in \FF_p^\times$.

Observing $\rho(\fg)=1$, we consider the $p$-subalgebra
\[ \fh := \fg_{\alpha}\!\oplus\!\ft\!\oplus\!\fg_{-\alpha}\]
along with the $\fh$-module $V:=\fg/\fh$. 

If $V=(0)$, then two cases arise. If $[\fg_\alpha,\fg_{-\alpha}] \subseteq \ker\alpha$, then $\fg/C(\fg) \cong \fb_{\fsl(2)}, \fb_{\fsl(2)}^{-1}$. Since
$\EE(2,\fb_{\fsl(2)}^{-1}) \ne \emptyset$, we must have $C(\fg) \ne (0)$ in the latter case. 
Alternatively, the restricted Lie algebra $\fg/C(\fg)$ is three-dimensional and not solvable, so that $\fg/C(\fg) \cong \fsl(2)$.

Hence we assume that $V\ne (0)$. This implies in particular that $p\ge 5$.

Suppose that $[\fg_\alpha,\fg_{-\alpha}] \subseteq \ker \alpha$. Then $\fh$ is solvable and such that 
\[ \fg_\alpha\!\oplus\!\fg_{-\alpha} \subseteq [\fh,\fh].\]
By Strade's Theorem (cf.\ \cite[(V.8.4)]{SF}), every simple $U_0(\fh)$-module has dimension a $p$-power. Since $\dim_kV \le p\!-\!3$, we conclude that 
every composition factor of $V$ is one-dimensional. Hence there is $i \in \{2,\ldots,p\!-\!2\}$ and a root vector $v \in \fg_{i\alpha}\!\smallsetminus\{0\}$ such that 
$(\fg_\alpha\!+\!\fg_{-\alpha})\dact v \subseteq \fh$. Thus, if $\fg_\alpha\dact v \ne (0)\ne \fg_{-\alpha}\dact v$, then 
$i\!+\!1, i\!-\!1 \in \{-1,0,1\}$, so that $i=0$, a contradiction. It follows that $v \in C_\fg(\fg_{\alpha})$ or $v \in C_\fg(\fg_{-\alpha})$, which contradicts 
Proposition \ref{CR1}(2).

We consider the factor algebra $\fg':=\fg/C(\fg)$ along with the canonical projection $\pi : \fg \lra \fg'$ and recall that $C(\fg)=\ker \alpha$. Then there is a toral element $t \in \ft$ such that
\[ \fg' = k\pi(t)\!\oplus\!\bigoplus_{i \in \FF_p^\times}\fg'_{i\alpha},\]
where $\pi : \fg_{i\alpha} \lra \fg'_{i\alpha}$ is an isomorphism and $\alpha$ is identified with the induced map $\ft/\ker\alpha \lra k$. Then we have
\[ C_{\fg'}(\fg'_\alpha) = \bigoplus_{i \in \FF_p}C_{\fg'}(\fg'_{\alpha})_{i\alpha}.\]
Since $[\fg'_\alpha, \fg'_{-\alpha}] \ne (0)$, it follows from Proposition \ref{CR1} that
\[ C_{\fg'}(\fg'_{\alpha}) = \fg'_\alpha.\]
Let $\fh':= \pi(\fh)$. Note that $\fh'=\fh/C(\fg)$ is isomorphic to $\fsl(2)$, where we pick $\pi(t)$ such that $\alpha(\pi(t))=2$. Given $x \in V(\fh')\!\smallsetminus\!
\{0\}$, there is $g \in \Aut_p(\fh')$ such that $x=g(x_\alpha)$. Setting $t':=g(\pi(t))$, the inverse image $\ft'=\pi^{-1}(kt')$ is a torus of dimension $\mu(\fg)$. The 
arguments above then show that $C_{\fg'}(x) = kx$. 

Recall that $\fg'$ is an $\fh'$-module of dimension $\le p$ such that $\fh' \subseteq \fg'$. In view of \cite{Pr}, every indecomposable $U_0(\fh')$-module of 
dimension $<p$ is simple. Thus, the $U_0(\fh')$-module $\fg'$ is either indecomposable or semi-simple. In the latter case, there exists an $\fh'$-submodule 
$W' \subseteq \fg'$ such that $\fg' = \fh'\oplus W'$. Since $V\ne (0)$, we have $W' \ne (0)$ and there is a weight vector $w' \in W'$ such that $[\fg'_\alpha,w']=0$, 
which contradicts $C_{\fg'}(\fg'_\alpha)=\fg'_\alpha$. Hence $\fg'$ is indecomposable and not simple, whence $\dim_k\fg'=p$. 

Let $x \in V(\fh')\!\smallsetminus\!\{0\}$. Since $C_{\fg'}(x)=kx$, it follows that $\fg'$ is a cyclic, projective $U_0(kx)$-module. A consecutive application of 
\cite[(3.2)]{FP87} and \cite[(1.4)]{FP86} then shows that $\fg'$ is a projective $U_0(\fh')$-module. Since $\dim_k \fg' = p$, we conclude that $\fg'$ is the Steinberg 
module for $U_0(\fh')$ and hence simple, a contradiction. 

As a result, we have $V=(0)$ and $\fg/C(\fg)$ has the asserted structure. \end{proof}

\bigskip

\begin{Example} Given $\fg:=\fb_{\fsl(2)}^{-1}$ as above, we consider the alternating form $\lambda : \fg\!\times\!\fg \lra k$ such that
\[ \lambda(x,y)=1 \ , \ \lambda(t,\fg)=0.\]
By definition, its Koszul differential $\partial(\lambda) \in \bigwedge^3(\fg)^\ast$ is an alternating $3$-form, which is uniquely determined by 
$\partial(\lambda)(t,x,y) \in k$. Direct computation shows that $\partial(\lambda)(t,x,y)=0$. Hence there is a one-dimensional non-split central extension 
$\fg^\lambda:=\fg\!\oplus\!kz$ of $\fg$ such that 
\[ [(u,\alpha z),(v,\beta z)] := ([u,v], \lambda(u,v)z)\]
for all $(u,\alpha z),(v,\beta z) \in \fg^\lambda$. The algebra $g^\lambda$ affords a $p$-map, given by
\[ (t,0)^{[p]} = (t,0) \ , \ (x,0)^{[p]}=(0,0) = (y,0)^{[p]} \ , \ (0,z)^{[p]}=(0,z).\]
Note that $\fg^\lambda$ is generically toral with $\mu(\fg^\lambda)=2=1\!+\dim_kC(\fg^\lambda)$.

The subspace $\fh := k(x,0)\!\oplus\!k(y,0)\!\oplus\!k(0,z)$ is a $p$-ideal of $\fg^\lambda$, which is isomorphic to the Heisenberg algebra with toral center, and 
such that $\fg/\fh$ is a torus. We therefore obtain $\EE(2,\fg^{\lambda}) = \EE(2,\fh) = \emptyset$, while $\fg^{\lambda}/C(\fg^\lambda) \cong \fb_{\fsl(2)}^{-1}$.
\end{Example} 

\bigskip
\noindent
Given $\ft \in \mstor(\fg)$ and $\alpha \in R_\ft$, we let
\[\fg[\alpha] := \bigoplus_{i \in \FF_p}\fg_{i\alpha}\]
be the \textit{$1$-section} of $\fg$ defined by $\alpha$. Note that $\fg[\alpha]$ is a $p$-subalgebra of $\fg$ such that $\mu(\fg[\alpha])=\mu(\fg)$.

\bigskip

\begin{Lem} \label{CR3} If $\fg$ is generically toral with $\msrke(\fg)=1$, then the following statements hold:
\begin{enumerate}
\item If $p\!=\!3$ and $r(\fg)\ge 2$, then $R_\ft\cup\{0\}$ is an $\FF_3$-vector space.
\item If $p\!\ge\! 5$, then 
\begin{enumerate}
\item $\rho(\fg) = 1$, and 
\item $\mu(\fg)=\dim_kC(\fg)\!+\!1$. \end{enumerate} \end{enumerate} \end{Lem}

\begin{proof} Let $\ft \in \mstor(\fg)$ be a torus of maximal dimension and let $\alpha \in R_\ft$. 

Suppose there is $\beta \in R_\ft\!\smallsetminus\!\FF_p\alpha$. We pick $x_\alpha \in \fg_\alpha\!\smallsetminus\!\{0\}$ and $x_\beta \in \fg_\beta\!
\smallsetminus\!\{0\}$. Proposition \ref{CR1}(2) implies that the adjoint representations of $x_\alpha$ and $x_\beta$ induce injective linear maps
\[ (\ast) \ \ \ \ \ f_{(\alpha,i,j)} : \fg_{i\alpha+j\beta} \lra \fg_{(i+1)\alpha+j\beta} \ \ ;  \ \ f_{(\beta,i,j)} : \fg_{i\alpha+j\beta} \lra \fg_{\alpha+(j+1)\beta} \ \ \ \ \ (i,j)\in \FF_p^2\]
unless $(\gamma,i,j) \in \{(\gamma,0,0), (\alpha, 1,0), (\beta, 0,1)\}$. 

(1) Let $p\!=\!3$. Let $\beta \in R_\ft\! \smallsetminus\! \FF_3\alpha$. In view of ($\ast$), application of suitable $f_{(\beta,1,j)}$ and $f_{(\alpha,i,1)}$ implies $\{\alpha,
\alpha\!+\!\beta, \alpha\!-\!\beta\} \cup \{\beta,\, \beta\!-\!\alpha\} \subseteq R_\ft$. Using first $f_{(\alpha,1,-1)}$, and then $f_{(\alpha,-1,-1)}$, $f_{(\beta,-1-1)}$ 
we thus obtain $\{-(\alpha\!+\beta), -\beta,-\alpha\} \subseteq R_\ft$, so that $\FF_3\alpha\!\oplus\!\FF_3\beta \subseteq R_\ft\cup\{0\}$. Hence $R_\ft\cup\{0\}$ is an 
$\FF_3$-vector space unless $R_\ft \subseteq \FF_3\alpha =\{0,\alpha,-\alpha\}$. As $r(\fg)\ge 2$, our assertion follows.

(2) Let $\alpha \in R_\ft$ be a root such that $\dim_k\fg_\alpha=\rho(\fg) \ge 2$. Then Proposition \ref{CR1}(4) yields $R_\ft\subseteq \{\alpha,-\alpha, -2\alpha\}$. 
As $p\ge 5$, $2\alpha$ is not a root, so that $\fg_{\alpha} \subseteq C_\fg(x_{\alpha})$ for all $x_{\alpha} \in \fg_{\alpha}\!\smallsetminus\!\{0\}$. This, however, 
contradicts Proposition \ref{CR1}(2). 

If $\beta \in R_\ft\!\smallsetminus\!\FF_p\alpha$, then the map
\[ f_{(\alpha,1,2)}\circ f_{(\beta,1,1)} \circ f_{(\beta,1,0)} :\fg_{\alpha} \lra \fg_{2(\alpha+\beta)}\]
is injective, so that $\{\alpha\!+\!\beta,2(\alpha\!+\!\beta)\} \subseteq R_\ft$. The $1$-section $\fg[\alpha\!+\!\beta]$ is a $p$-subalgebra such that $\rho(\fg[\alpha\!+\!\beta])=1=
\msrke(\fg[\alpha\!+\!\beta])$, while $\mu(\fg[\alpha\!+\!\beta])=1\!+\!\dim_kC(\fg[\alpha\!+\!\beta])$. Lemma \ref{CR2} thus yields $2(\alpha\!+\!\beta)=-(\alpha\!+\!\beta)$. 
As $p\!\ge\!5$, we obtain a contradiction, whence $R_\ft \subseteq \FF_p\alpha$. This readily implies (2). \end{proof}

\bigskip
\noindent
The foregoing results readily yield Theorem A:

\bigskip

\begin{Thm} \label{CR4} Suppose that $p\!\ge\!5$. Let $(\fg,[p])$ be generically toral and such that $\msrke(\fg)=1$. Then we have
\[ \fg/C(\fg) \cong \fsl(2), \fb_{\fsl(2)}, \fb^{-1}_{\fsl(2)},\]
with $C(\fg)\ne (0)$ in the latter case. \end{Thm}

\begin{proof} In view of Lemma \ref{CR3}(2) we have $\rho(\fg)=1$ and $\mu(\fg)=1\!+\!\dim_kC(\fg)$. The result now follows from Lemma \ref{CR2}. \end{proof} 

\bigskip

\begin{Remark} It is well-known that the cohomology groups $\HH^2(\fg',k)$ vanish in case $\fg' \cong \fsl(2),\,\fb_{\fsl(2)}$. Hence the first two types of algebras 
mentioned in Theorem \ref{CR4} are direct products $\fg = \fg'\!\times\!C(\fg)$ of the restricted Lie algebras $\fg'$ and a toral centers, and with $p$-maps given by
\[(x,c)^{[p]} = (x^{[p]},c^{[p]}\!+\!f(x))  \ \ \ \ \ (x,c) \in \fg,\]
where $f : \fg' \lra C(\fg)$ is $p$-semilinear. \end{Remark}

\bigskip
\noindent
Using sandwich elements, the first author obtained the first part of the following result in his doctoral dissertation \cite{Ch16}:

\bigskip

\begin{Cor} \label{CR5} Suppose that $p\!\ge\!5$ and let $(\fg,[p])$ be such that $\msrke(\fg)=1$.
\begin{enumerate}
\item If $\fg$ is centerless, then $\fg \cong \fsl(2), \fb_{\fsl(2)}$.
\item If $\fg$ is not solvable, then $\fg/C(\fg) \cong \fsl(2)$.\end{enumerate}
\end{Cor}

\begin{proof} (1) In view of Corollary \ref{LR2}, the Lie algebra $\fg$ is generically toral. Since $C(\fg)=(0)$, Theorem \ref{CR4} implies the result.

(2) Using Corollary \ref{LR2} again, we see that $\fg$ is generically toral. As $\fg/C(\fg)$ is not solvable, it follows from Theorem \ref{CR4} that
$\fg/C(\fg) \cong \fsl(2)$. \end{proof} 

\bigskip

\begin{Cor} \label{CR6} Suppose that $p\!\ge\!5$. Let $(\fg,[p])$ be a restricted Lie algebra. If $r(\fg)\!\ge\! 3$, then $\EE(2,\fg) \ne \emptyset$. \end{Cor}

\begin{proof} Suppose that $\EE(2,\fg)=\emptyset$. If $\msrke(\fg)=0$, then $\fg$ is a torus and $r(\fg)=0$, a contradiction. Hence $\msrke(\fg)=1$, and Proposition \ref{LR1}
shows that $\fg$ is generically toral. Theorem \ref{CR4} thus yields $r(\fg)\le 2$, a contradiction. \end{proof}

\bigskip

\subsection{The case $p\!=\!3$} It turns out that several results of the foregoing subsection do not hold for small $p$. Throughout this section, we assume that
$(\fg,[p])$ is a restricted Lie algebra, defined over an algebraically closed field $k$ of characteristic $\Char(k)\!=\!3$.

\bigskip

\begin{Lem} \label{CR7} Let $\fg$ be generically toral with $\msrke(\fg)=1$. If $\rho(\fg)\!\ge\!2$, then we have
\begin{enumerate} 
\item $\mu(\fg)\ge \rho(\fg)$ and $C(\fg)\ne (0)$.
\item $\fg$ is solvable.
\end{enumerate} \end{Lem} 

\begin{proof} We choose $\ft \in \mstor(\fg)$ and let $\alpha \in R_\ft$ be a root such that $\dim_k\fg_\alpha = \rho(\fg)$. Proposition \ref{CR1}(4) ensures that 
$R_\ft \subseteq \{\alpha,-\alpha\}$ and we therefore write
\[ \fg = \ft \oplus \fg_\alpha\!\oplus\!\fg_{-\alpha}.\]

(1) If $\fg_{-\alpha}=(0)$, then $\fg_{\alpha} \subseteq C_\fg(x_{\alpha})$ for $x_{\alpha} \in \fg_{\alpha}\!\smallsetminus\!\{0\}$, which contradicts 
Proposition \ref{CR1}(2). Given $x_{-\alpha} \in \fg_{-\alpha}\!\smallsetminus\!\{0\}$, Proposition \ref{CR1}(2) implies that the map
\[ \fg_{\alpha} \lra \ft \ \ ;  \ \ v \mapsto [x_{-\alpha}, v] \]
is injective, so that $\mu(\fg) = \dim_k \ft \ge \dim_k\fg_{\alpha} = \rho(\fg)$. Hence $\dim_kC(\fg) = \mu(\fg)\!-\!1 \ne 0$. 

(2) Suppose that $[\fg_{\alpha},\fg_{-\alpha}] \not \subseteq \ker \alpha$. We consider the Lie algebra $\fg':= \fg/\ker\alpha$. Writing $\ft = 
kt\!\oplus\!\ker\alpha$ for some toral element $t \in \ft$, we have
\[ \fg' = kt'\!\oplus\!\fg'_{\alpha}\!\oplus\! \fg'_{-\alpha},\]
and our current assumption in conjunction with Lemma \ref{Gt2} provides $x'_{\alpha} \in \fg'_{\alpha}$ and $x'_{-\alpha} \in \fg'_{-\alpha}$ such that
\[ \fh' := kt'\!\oplus \!kx'_{\alpha}\!\oplus\!kx'_{-\alpha}\]
is a $p$-subalgebra such that $\fh'\cong \fsl(2)$. We pick $t'$ such that $\alpha(t')=2$ and $x'_{\alpha},x'_{-\alpha}$ such that $[x'_{\alpha},
x'_{-\alpha}]=t'$. Note that the simple $3$-dimensional module $\fh'$ is the Steinberg module for the restricted enveloping algebra $U_0(\fh')$ of $\fh'$. Hence 
it is injective, and there is an $\fh'$-submodule $W' \subseteq \fg'$ such that
\[ \fg' = \fh'\!\oplus\!W'.\]
Since $0$ is not a weight of $W'$, every composition factor of $W'$ also enjoys this property. Hence every such factor is isomorphic to the standard module 
$L(1)$. In view of $\Ext^1_{U_0(\fsl(2))}(L(1),L(1))=(0)$ (cf.\ \cite[Thm.1]{Po68} which also holds for $p\!=\!3$), it follows that $W'\cong L(1)^n$ for some 
$n \in \NN_0$. Suppose that $n\ge 1$. The multiplication provides a surjective $\fh'$-linear map $\fh'\!\otimes_k\!\fg' \lra \fg'$. Thanks to \cite{BO}, $L(2)\!\otimes_k\!L(1)\cong P(1)$ is the projective cover of $L(1)$, and we obtain a 
surjection $P(1)^n \lra \fg'$ of $U_0(\fh')$-modules. As $\fg'$ is a semi-simple $U_0(\fh')$-module, we get a surjection $L(1)^n \lra \fg'$, which contradicts 
$L(2) \cong \fh' \subseteq \fg'$. Hence $n=0$, and we have $\dim_k\fg_{\alpha}=1$, a contradiction.

We conclude that $[\fg_{\alpha},\fg_{-\alpha}] \subseteq \ker\alpha$. Thanks to Lemma \ref{Gt2}, $L:=\fg_\alpha\cup\fg_{-\alpha}\cup\ker\alpha$ is a Lie subset of $\fg$ that acts on
$\fg$ by nilpotent transformations. The Engel-Jacobson Theorem thus shows that the $p$-ideal $\fn:=\langle L \rangle$ is nilpotent. Hence $[\fg,\fg] \subseteq \fn$ enjoys the same property, 
so that $\fg$ is solvable. \end{proof}

\bigskip

\begin{Example} Consider the $5$-dimensional vector space
\[ \fh := T(\fh)\!\oplus\! \fh_1\!\oplus\!\fh_{-1},\]
where $T(\fh)=kt_1\!\oplus\!kt_2$, $\fh_1:= kx_1\!\oplus\!kx_2$, and $\fh_{-1}=ky$. We define a Lie bracket via
\[ [T(\fh),\fh]=(0) \ \ ; \ \  [x_1,x_2] = y \ \ ; \ \ [x_i,y]=t_i \ \ \ \ \ \ \ \ \ 1 \le i \le 2,\]
and a $p$-map by means of
\[ \fh_i^{[p]}=\{0\} \ \ \text{and} \ \ t_i^{[p]} = t_i \ \ \ \ 1 \le i \le 2.\]
Direct computation shows that $V(\fh) = \fh_1\!\oplus\!\fh_{-1}$ and $\EE(2,\fh)=\emptyset$, so that $\msrke(\fh)=1$. 

Now set $\fh_0:= T(\fh)$. Then 
\[ d : \fh \lra \fh \ \ ; \ \ d|_{\fh_i} = i\id_{\fh_i}\]
is a derivation of $\fh$ such that $d(h^{[p]})= (\ad h)^{p-1}(d(h))$ for all $h \in \fh$ and $d^p=d$. Hence the semidirect sum $\fg:= kd\!\ltimes\!\fh$ is a generically 
toral restricted Lie algebra with $\rho(\fg)=2$ and $\EE(2,\fg)=\EE(2,\fh)=\emptyset$. \end{Example}

\bigskip
\noindent
Let $(\fg,[p])$ be generically toral with maximal torus $\ft \in \mstor(\fg)$. We say that a root $\alpha \in R_\ft$ is {\it solvable}, if its $1$-section $\fg[\alpha]$ is 
solvable. 

\bigskip

\begin{Prop} \label{CR8} Let $(\fg,[p])$ be a generically toral restricted Lie algebra such that $\msrke(\fg)=1$. If $\ft \in \mstor(\fg)$ is such that $R_\ft$ 
contains a non-solvable root, then $\fg/C(\fg) \cong \fsl(2)$. \end{Prop}

\begin{proof} Let $\alpha \in R_\ft$ be non-solvable and suppose there in $\beta \in R_\ft\!\smallsetminus\!\{\alpha,-\alpha\}$. Owing to Proposition \ref{CR1}(4), 
we have $\rho(\fg)=1$ and Lemma \ref{CR3}(1) yields
\[ \FF_3\alpha\!\oplus\!\FF_3\beta \subseteq R_\ft \cup \{0\}.\] 
We consider the $p$-subalgebra
\[ \fh := \bigoplus_{(i,j) \in \FF_p^2} \fg_{i\alpha+j\beta}.\]
Then $\fh$ is generically toral with $\mu(\fh)=\mu(\fg)$, $\rho(\fh)=1$ and of $p$-rank $\msrke(\fh)=1$. We consider the root space decomposition
\[ \fh = \ft\!\oplus\!\bigoplus_{\gamma \in R'_\ft} \fh_\gamma,\]
where $R'_\ft = (\FF_3\alpha\!\oplus\!\FF_3\beta)\!\smallsetminus\{0\} \subseteq R_\ft$. As $\fh[\gamma] = \fg[\gamma]$ for all
$\gamma \in R'_\ft$, $\alpha$ is also a non-solvable root of $R'_\ft$. Lemma \ref{CR2} now implies $\fh[\alpha]/\ker\alpha \cong \fsl(2)$. 

Let $I \unlhd \fh$ be an ideal such that $C(\fh)\subseteq I$. Then the space
\[ I[\alpha] := \bigoplus_{i\in \FF_3} I_{i\alpha}\]
is an ideal of $\fh[\alpha]$. As $\fh[\alpha]/\ker\alpha$ is simple, this implies
\[ I[\alpha] \subseteq \ker \alpha \ \text{or} \ I[\alpha]=\fh[\alpha].\]
If $I[\alpha]=\fh[\alpha]$, then $\ft \subseteq I[\alpha] \subseteq I$ and $I=\fh$.

Alternatively, $I[\alpha] \subseteq \ker \alpha$, so that $I_\alpha=(0)=I_{-\alpha}$. Let $\gamma \in R'_\ft\!\smallsetminus\!\{\alpha,-\alpha\}$ be such that 
$I_{\gamma} \ne (0)$. As $R'_\ft = (\FF_3\alpha\!\oplus\!\FF_3\beta)\!\smallsetminus\{0\}$, there is $x \in \fh_{\alpha-\gamma}\!\smallsetminus\!\{0\}$. Thanks
to Proposition \ref{CR1}(2), the map 
\[ \fh_{\gamma} \lra \fh_{\alpha} \ \ ; \ \ v \mapsto [x,v]\]
is injective, while $\fh_\gamma = I_\gamma$ and $[x, I_\gamma] \subseteq I_\alpha=(0)$. Hence we have $I_\gamma = (0)$ for all $\gamma \in R'_\ft$.
This implies $I\subseteq \ft$ as well as $[I,\fh_\gamma] \subseteq I\cap\fh_\gamma=(0)$ for all $\gamma \in R'_\ft$, whence $I\subseteq \bigcap_{\gamma \in
R'_\ft} \ker\gamma = C(\fh)$. We thus have $I=C(\fh)$ whenever $I[\alpha]\subseteq \ker\alpha$.

We consider the restricted Lie algebra $\fh':= \fh/C(\fh)$. By the above, $\fh'$ is simple and Lemma \ref{Gt1} ensures that $\fh'$ is generically toral. Since 
$C(\fh)=\ker\alpha\cap\ker\beta$, we have $\mu(\fh')=2$. Let $\pi : \fh \lra \fh'$ be the canonical projection and put $\ft':= \pi(\ft) \cong \ft/C(\fh)$. Every root 
$\gamma \in R'_\ft$ gives rise to a root $\gamma' \in R'_{\ft'}$ such that $\gamma'\circ\pi=\gamma$. Thus, $R'_{\ft'}\cup\{0\} = \FF_3\alpha'\!\oplus\!\FF_3\beta'$. 
It follows that
\begin{enumerate}
\item[(a)] $\dim_k\fh'=10$, and
\item[(b)] $\fh'[\alpha']/\ker\alpha' \cong \fsl(2)$, so that $[\fh'_{\alpha'},\fh'_{-\alpha'}] \ne (0)$.
\end{enumerate} 
Since $\fh'$ is simple, we have $\fh'=[\fh',\fh']$ and there exists $\gamma' \in R'_{\ft'}$ such that $[\fh'_{\gamma'},\fh'_{-\gamma'}] \not \subseteq [\fh'_{\alpha'},
\fh'_{-\alpha'}]$. Thus, $\ft'= [\fh'_{\gamma'},\fh'_{-\gamma'}]\! \oplus \! [\fh'_{\alpha'},\fh'_{-\alpha'}]$ and the proof of Lemma \ref{CR3}(1) implies that $\fh'$ is generated 
by $\fh'_{\alpha'}\!\oplus\!\fh'_{\gamma'}$. Hence $\fh'$ is freely generated, and a consecutive application of Lemma \ref{CE2} and Lemma \ref{CE1} yields 
$\msrke(\fh')=\msrke(\fh)=1$. Now Lemma \ref{CenSat2} implies
\[ \dim_k\fh' \le 3(2\mu(\fh')\!-\!1)=9,\]
a contradiction. 

It follows that $R_\ft \subseteq \{\alpha,-\alpha\}$, so that $\fg=\fg[\alpha]$ and $\fg/C(\fg)\cong \fsl(2)$, by Lemma \ref{CR2}. \end{proof}  

\bigskip

\begin{Lem} \label{CR9} Let $(\fg,[p])$ be generically toral of $p$-rank $\msrke(\fg)=1$. Suppose that $\ft \in \mstor(\fg)$ is such that every root $\alpha \in R_\ft$
is solvable. Then we have
\[ [\fg_\alpha,\fg_{-\alpha}] \subseteq C(\fg)\]
for all $\alpha \in R_\ft$. In particular, if $C(\fg)=(0)$, then $\fg \cong \fb_{\fsl(2)}$. \end{Lem}

\begin{proof} Suppose that there are $\alpha, \beta \in R_\ft$ are such that $\beta \in R_\ft\!\smallsetminus\!\{\alpha,-\alpha\}$. Thanks to Proposition \ref{CR1}(4), 
we then have $\rho(\fg)=1$.

We first assume that $R_\ft \subseteq \FF_3\alpha\!\oplus\!\FF_3\beta$. Lemma \ref{CR3}(1) then implies that $R_\ft\cup\{0\}=\FF_3\alpha\!\oplus\!\FF_3\beta$. We consider the Lie algebra 
$\fg':=\fg/C(\fg)$ with its corresponding root space decomposition relative to $\ft'\cong \ft/C(\fg)$.

If $\dim_k\sum_{\lambda' \in R_{\ft'}}[\fg'_{\lambda'},\fg'_{-\lambda'}] = 2$, then there are $\FF_3$-independent roots $\gamma',\delta' \in R_{\ft'}$ such that
\[ \ft' = [\fg'_{\gamma'},\fg'_{-\gamma'}]\!\oplus\![\fg'_{\delta'},\fg'_{-\delta'}].\]
Thus, $\fg'$ is freely generated by $\fg'_{\gamma'}\!\oplus\!\fg'_{\delta'}$, so that Lemma \ref{CE2} and Lemma \ref{CE1} yield $\msrke(\fg')=1$. Since $\dim_k\fg'=10$, this contradicts Lemma \ref{CenSat2}.

We therefore have $\dim_k\sum_{\lambda' \in R_{\ft'}}[\fg'_\lambda,\fg'_{-\lambda}] \le 1$. Hence there is $\lambda' \in R_{\ft'}$ such that
\[ (\ast) \ \ \ \ \ \ \ \  [\fg'_{\gamma'},\fg'_{-\gamma'}] \subseteq [\fg'_{\lambda'},\fg'_{-\lambda'}] \]
for every $\gamma' \in R_{\ft'}.$

Let $\gamma \in R_\ft$ be a root, $\gamma' \in R_{\ft'}$ be the corresponding root. Since $\gamma$ is solvable and $\fg'_{-\gamma'} \ne (0)$, it follows from Lemma \ref{CR2} that 
$\fg[\gamma]/\ker\gamma \cong \fb_{\fsl(2)}^{-1}$. Thus, $[\fg_{\gamma},\fg_{-\gamma}] \subseteq \ker\gamma$. The inclusion ($\ast$) now implies 
\[ [\fg'_{\gamma'},\fg'_{-\gamma'}] \subseteq \ker\gamma'\cap\ker\lambda' = (0)\]
for all $\gamma' \in R_{\ft'}\!\smallsetminus\FF_3\lambda'$. Hence there is a basis $\{\gamma',\lambda'\}$ of $R_{\ft'}\cup\{0\}$ such that
\[ [\fg'_{\sigma'},\fg'_{-\sigma'}] =(0)\]
for $\sigma' \in \{\gamma',\lambda'\}$. Now let $i,j \in \{1,-1\}$. Then we have
\begin{eqnarray*} [\fg'_{i\gamma'+j\lambda'},\fg'_{-i\gamma'-j\lambda'}]  & = & [\fg'_{i\gamma'+j\lambda'},[\fg'_{-i\gamma'},\fg'_{-j\lambda'}]] \subseteq 
 [[\fg'_{i\gamma'+j\lambda'},\fg'_{-i\gamma'}],\fg'_{-j\lambda'}] \!+\! [\fg'_{-i\gamma'},[\fg'_{i\gamma'+j\lambda'},\fg'_{-j\lambda'}]] \\
 & \subseteq & [\fg'_{j\lambda'},\fg'_{-j\lambda'}] \!+\! [\fg'_{-i\gamma'},\fg'_{i\gamma'}] = (0),
 \end{eqnarray*}
 whence
 \[  [\fg'_{\lambda'},\fg'_{-\lambda'}] = (0) \ \ \ \ \ \ \forall \ \lambda' \in R_{\ft'}.\]
 Consequently, 
 \[  [\fg_{\lambda},\fg_{-\lambda}] \subseteq  C(\fg) \ \ \ \ \forall \ \lambda \in R_{\ft}.\]  
Let $\alpha \in R_\ft$. If $R_\ft \subseteq \FF_3\alpha$, and $\rho(\fg)=1$, then our assertion follows directly from Lemma \ref{CR2}. If $\rho(\fg)\ge 2$, then the proof of Lemma \ref{CR7}
gives $[\fg_\alpha,\fg_{-\alpha}]\subseteq C(\fg)$.

Alternatively, let $\beta \in R_\ft\!\smallsetminus\! \FF_3\alpha$. Considering the $p$-subalgebra
\[\fg[\alpha,\beta] := \bigoplus_{(i,j) \in \FF_3^2}\fg_{i\alpha+j\beta},\] 
we obtain from the above
 \[ [\fg_\alpha,\fg_{-\alpha}]\subseteq \ker\alpha\cap\ker\beta.\]
Thus, $[\fg_\alpha,\fg_{-\alpha}] \subseteq \bigcap_{\beta \in R_\ft} \ker\beta = C(\fg)$, as desired. 

Let $\alpha \in R_\ft$. The assumption $C(\fg)=(0)$ now yields $\fg_{-\alpha} \subseteq C_\fg(\fg_{\alpha})$, so that Proposition \ref{CR1} implies $\fg_{-\alpha}
=(0)$. By Lemma \ref{CR3}(1), this can only happen for $R_\ft=\{\alpha\}$, so that Lemma \ref{CR2} yields the assertion. \end{proof}  

\bigskip

\subsection{Lie algebras of characteristic $p\!\ge\! 3$.}
We provide results that hold for $p\!\ge\!3$.

\bigskip

\begin{Thm} \label{CR10} Let $(\fg,[p])$ be a restricted Lie algebra such that $\msrke(\fg)=1$. Then the following statements hold:
\begin{enumerate}
\item If $\fg$ is centerless, then $\fg\cong \fb_{\fsl(2)}, \fsl(2)$.
\item If $\fg$ is perfect, then $\fg \cong \fsl(2)$.
\end{enumerate} \end{Thm}

\begin{proof} (1) In view of Corollary \ref{LR2}, the algebra $\fg$ is generically toral. Let $\ft \in \mstor(\fg)$ be a torus of maximal dimension. If $p\!\ge\! 5$, then Corollary \ref{CR5} 
implies our assertion. If $p\!=\!3$ and every root $\alpha \in R_\ft$ is solvable, then Lemma \ref{CR9} yields $\fg\cong\fb_{\fsl(2)}$. Alternatively, Proposition \ref{CR8} implies
$\fg \cong \fsl(2)$. 

(2) Since $\fg$ is perfect, it is not solvable, so that Corollary \ref{LR2} shows that $\fg$ is generically toral. For $p\!\ge\!5$, the result is a consequence of Theorem \ref{CR4}. 

Let $p\!=\!3$ and consider $\ft \in \mstor(\fg)$. If every root $\alpha \in R_\ft$ is solvable, then the Lemma \ref{CR9} implies
\[ \fg = [\fg,\fg] = \sum_{\alpha \in R_\ft}[\fg_\alpha,\fg_{-\alpha}]\!\oplus \!\bigoplus_{\alpha \in R_\ft}\fg_{\alpha} \subseteq C(\fg) \!\oplus \!\bigoplus_{\alpha \in R_\ft}\fg_{\alpha}.\]
Thus, $\ft=C(\fg)$, so that $\fg$ is nilpotent, a contradiction. 

Hence there is a non-solvable root $\alpha \in R_\ft$, and Proposition \ref{CR8} yields $\fg/C(\fg) \cong \fsl(2)$. Since $\HH^2(\fsl(2),k)=(0)$,
we have $\fg \cong \fsl(2)\!\oplus\!C(\fg)$, whence $\fg=[\fg,\fg] = \fsl(2)$. \end{proof} 

\bigskip

\begin{Cor} \label{CR11} Let $(\fg,[p])$ be a restricted Lie algebra of $p$-rank $1$ and such that $T(\fg)=(0)$. Then
\[ \fg \cong \fsl(2), \, \fb_{\fsl(2)}, \, kt\!\ltimes\!(kx)_p,\]
where $t$ is toral and $x\ne 0$ is $p$-nilpotent. \end{Cor}

\begin{proof} Since $T(\fg)=(0)$, the center $C(\fg)$ is unipotent. If $C(\fg)=(0)$, then Theorem \ref{CR10} yields $\fg \cong \fsl(2), \, \fb_{\fsl(2)}$. Alternatively,
$V(C(\fg)) \ne \{0\}$ and since $\rk_p(\fg)=1$, it follows that $V(\fg)=V(C(\fg))$ is a line. Hence
\[ \fg \cong kt\!\ltimes\!(T(\fg)\!\oplus\!(kx)_p,\]
with $t$ toral and $x$ $p$-nilpotent, cf.\ \cite[(3.2),(4.3)]{Fa95}. Since $T(\fg)=(0)$, $\fg$ is an algebra of the third type. \end{proof} 

\bigskip

\begin{Remark} In view of Brauer's results for finite groups \cite{Br76,Br79}, one would hope to determine $\fg/T(\fg)$ for any restricted Lie algebra of $p$-rank $\msrke(\fg)=1$.
Contrary to finite groups, where $\msrke(G/O_{p'}(G))=\msrke(G)$, the example of the Heisenberg algebra with toral center shows that the $p$-rank of 
$\fg/T(\fg)$ may exceed $\msrke(\fg)$. \end{Remark}

\bigskip

\section{Finite group schemes of $p$-rank $\msrke(\cG)\le 1$} \label{S:FG}
In this concluding section, we consider finite group schemes of $p$-rank $\le 1$. The reader is referred to \cite{Ja06} and \cite{Wa79} for basic facts concerning
algebraic group schemes. Given a finite group scheme $\cG$ over $k$, we denote by $k[\cG]$ and $k\cG:=k[\cG]^\ast$, the coordinate ring and the group algebra
(the algebra of measures) of $\cG$, respectively. In what follows, all subgroup schemes are supposed to be closed.

Following \cite{Fa14}, we refer to an abelian finite group scheme $\cE$ as \textit{elementary abelian}, provided there exist subgroup schemes $\cE_1, \ldots, \cE_n \subseteq \cE$ such that
\begin{enumerate}
\item[(a)] $\cE = \cE_1\cdots \cE_n$, and
\item[(b)] $\cE_i \cong \GG_{a(r_i)}, \ZZ/(p)$. \end{enumerate}
Here $\GG_{a(r)}$ denotes the $r$-th Frobenius kernel of the additive group $\GG_a = \Spec_k(k[T])$, while $\ZZ/(p)$ refers to the reduced group scheme, whose group of $k$-rational points is the cyclic group $\ZZ/(p)$. In view of \cite[(6.2.1),(6.2.2)]{Fa14}, a finite group scheme $\cE$ is elementary abelian if
and only if $k\cE \cong k[T_1,\ldots,T_r]/(T_1^p,\ldots,T_r^p)$. We call $\msrke(\cE):=r$ the $p$-rank of $\cE$, so that $\dim_kk\cE = p^{\msrke(\cE)}$. 

\bigskip

\begin{Definition} Let $\cG$ be a finite group scheme. Then  
\[ \msrke(\cG) := \max\{\msrke(\cE) \ ; \ \cE \subseteq \cG \ \text{elementary abelian}\}\]
is called the \textit{$p$-rank} of $\cG$. \end{Definition}

\bigskip

\begin{Remarks} (1) Let $(\fg,[p])$ be a restricted Lie algebra. Then $\cG_\fg:=\Spec_k(U_0(\fg)^\ast)$ is an infinitesimal group of height $1$ such that
$k\cG_\fg = U_0(\fg)$. It follows from the above that $\msrke(\cG_\fg)=\msrke(\fg)$, cf.\ also \cite[(II,\S7, ${\rm n}^{\rm o}$4)]{DG70}.

(2) If $G$ is a finite group, then the $p$-rank of $G$ coincides with that of its associated reduced group scheme $\cG_G:=\Spec_k(kG^\ast)$.
\end{Remarks}

\bigskip

\begin{Example} By way of illustration, we begin by considering the case, where $\cG=G_r$ is the $r$-th Frobenius kernel of an algebraic group $G$. If $G$ is not a torus,
then the arguments employed in the proof of Theorem \ref{AL1} provide a connected unipotent subgroup $e_k \ne U\subseteq G$. According to \cite[(IV,\S4,3.4)]{DG70}
$U$ contains a subgroup of type $\GG_a$, so that $\GG_{a(r)} \subseteq G_r$ and $r = \msrke(\GG_{a(r)}) \le \msrke(G_r)$. Consequently, $\msrke(G_r) \le 1$ only
if $r=1$ or $G_r \cong \GG_{m(r)}^n$ is diagonalizable. Since $G_1$ corresponds to $\fg=\Lie(G)$, Theorem \ref{AL1} provides the structure of Frobenius kernels
of $p$-rank $1$. \end{Example}

\bigskip
\noindent
We recall basic features from the Friedlander-Pevtsova theory of $p$-points \cite{FPe05}. Let $\fA_p:=k[T]/(T^p)$ be the $p$-truncated polynomial ring in
one variable. Given a finite group scheme $\cG$, an algebra homomorphism $\alpha : \fA_p \lra k\cG$ is called a {\it $p$-point}, provided
\begin{enumerate}
\item[(P1)] the map $\alpha$ is left flat, and
\item[(P2)] there exists an abelian unipotent subgroup $\cU \subseteq \cG$ such that $\im \alpha \subseteq k\cU$. \end{enumerate}
We let $\pt(\cG)$ be the set of $p$-points of $\cG$. Every $\alpha \in \pt(\cG)$ defines an exact functor
\[ \alpha^\ast : \modd k\cG \lra \modd \fA_p\]
between the respective categories of finite-dimensional modules, which, in view of (P1), sends projectives to projectives. Hence the full subcategory
$\modd^{\alpha^\ast}k\cG$, with objects being those $\cG$-modules whose pull-back $\alpha^\ast(M)$ along $\alpha$ is projective, contains all projective 
$\cG$-modules. Two $p$-points $\alpha,\beta \in \pt(\cG)$ are {\it equivalent} ($\alpha\! \sim\! \beta$), if  $\modd^{\alpha^\ast}k\cG =  \modd^{\beta^\ast}k\cG$. 
We denote by $\msP(\cG):=\pt(\cG)/\!\sim$ the space of $p$-points. Thanks to \cite[(3.10)]{FPe05}, $\msP(\cG)$ is a noetherian topological space, whose closed
sets are of the form
\[ \msP(\cG)_M := \{ [\alpha] \in \msP(\cG) \ ; \ \alpha^\ast(M) \ \text{is not projective}\} \ \ \ \ \ \ \ \ (M \in \modd \cG).\]
Let $\cH \subseteq \cG$ be a subgroup scheme. The canonical inclusion $\iota : k\cH \lra k\cG$, defines a continuous map
\[ \iota_\ast : \msP(\cH) \lra \msP(\cG) \ \ ; \ \ [\alpha] \mapsto [\iota\circ\alpha].\]
Recall that every finite group scheme $\cG$ is a semi-direct product
\[ \cG = \cG^0\!\rtimes\!\cG_{\rm red}\]
of an infinitesimal normal subgroup $\cG^0$ and a reduced subgroup, see \cite[(6.8)]{Wa79}. Given $r \in \NN_0$ , we denote by $\cG_r=(\cG^0)_r$ the $r$-th
Frobenius kernel of $\cG$. In particular, we have $\cG_0=e_k$.

Let $\cG$ be a finite group scheme of $p$-rank $0$. Then $\msrke(\cG^0)=0=\msrke(\cG_{\rm red})$ and $p\! \nmid\! {\rm ord}(\cG(k))$. Moreover, \cite[(IV,\S3,(3.7)]{DG70} shows that $\cG^0$ is diagonalizable. 
By Nagata's Theorem \cite[(IV,\S3, (3.6)]{DG70}, these properties of $\cG$ are equivalent to $k\cG$ being semi-simple. In that case, we say that $\cG$ is \textit{linearly reductive}.  

\bigskip

\begin{Lemma} \label{IG1} Let $\cG$ be a finite group scheme of $p$-rank $\msrke(\cG)=1$.
\begin{enumerate}
\item If $\cE \subseteq \cG$ is elementary abelian, then $\cE \subseteq \cG_1\!\rtimes\!\cG_{\rm red}$.
\item We have $\msP(\cG)=\iota_\ast(\msP(\cG_1\!\rtimes\!\cG_{\rm red}))$. \end{enumerate} \end{Lemma}

\begin{proof} (1) We write $\cE = \cE^0\times \cE_{\rm red}$, with each factor being elementary abelian, cf.\ \cite[(6.2.2)]{Fa14}. Since $\cE^0$ does not contain
any reduced subgroup schemes $\ne e_k$, it follows that $\cE^0$ is a product of groups of type $\GG_{a(r_i)}$. In view of
\[ r_i = \msrke(\GG_{a(r_i)}) \le \msrke(\cE) \le 1,\]
each factor of $\cE^0$ is contained in $\cG_1$, whence $\cE^0 \subseteq \cG_1$. Consequently, \cite[(1.1)]{Fa06} yields $\cE \subseteq 
\cG_1\!\rtimes\!\cG_{\rm red}$.

(2) Let $x \in \msP(\cG)$. Thanks to \cite[(4.2)]{FPe05}, there exist $\alpha \in x$ and an elementary abelian subgroup $\cE \subseteq \cG$
such that $\im \alpha \subseteq \cE$. In view of (1), this shows that $x \in \iota_\ast(\msP(\cG_1\!\rtimes\!\cG_{\rm red}))$. \end{proof}  

\bigskip
\noindent
Given an infinitesimal group $\cG$, we denote by $\cM(\cG)$ the unique largest diagonalizable (multiplicative) normal subgroup of $\cG$.
In view of \cite[(7.7),(9.5)]{Wa79}, the group scheme $\cM(\cG)$ coincides with the multiplicative constituent of the center $\cC(\cG)$ of $\cG$.

For a commutative unipotent infinitesimal group scheme $\cU$, we let
\[ V_{\cU} : \cU^{(p)} \lra \cU\]
be the Verschiebung, cf.\ \cite[(IV,\S3,${\rm n}^{\rm o}$4),(II,\S7,${\rm n}^{\rm o}$1)]{DG70}. Following \cite{FV00}, we refer to
$\cU$ as being \textit{V-uniserial}, provided there is an exact sequence
\[\cU^{(p)} \stackrel{V_{\cU}}{\lra} \cU \lra \GG_{a(1)} \lra e_k.\]
If $\cU$ is $V$-uniserial, then the complexity $\cx_{\cU}(k)$ of the trivial $\cU$-module $k$ equals $1$, cf.\ \cite[(2.6)]{FV00}. Thus, if $\cE \subseteq \cU$ is
elementary abelian, then
\[ \msrke(\cE)=\cx_\cE(k) \le \cx_{\cU}(k) \le 1,\]
so that $\msrke(\cU) = 1$.   

The $r$-th Frobenius kernel of the multiplicative group $\GG_m := \Spec_k(k[X,X^{-1}])$ will be denoted $\GG_{m(r)}$. We let $T\subseteq \SL(2)$ be the standard
torus of diagonal matrices.

\bigskip

\begin{Proposition} \label{IG2} Let $\cG$ be an infinitesimal $k$-group of $p$-rank $\msrke(\cG)=1$. If $\cM(\cG)=e_k$, then
\[ \cG \cong \SL(2)_1T_r, \, \cU\!\rtimes\!\GG_{m(r)} \ \ \ (r\!\ge\!0), \]
where $\cU$ is a $V$-uniserial normal subgroup of $\cG$ and $\GG_{m(r)}$ acts faithfully on $\cU$. \end{Proposition}

\begin{proof} Let $\cM(\cG_1)$ be the multiplicative center of $\cG_1$. In view of \cite[(IV,\S3,(1.1))]{DG70}, $\cM(\cG_1)$ is a normal,
multiplicative subgroup of $\cG$, so that $\cM(\cG_1) \subseteq \cM(\cG)=e_k$. As a result, the toral radical $T(\fg)$ of the Lie algebra
$\fg:= \Lie(\cG)$ is trivial and Corollary \ref{CR11} yields $\fg \cong \fsl(2)$ or $\dim V(\fg)=1$.

Suppose that $\dim V(\fg)=1$, and let $V(\cG)$ be the variety $k$-rational points of the scheme of infinitesimal one-parameter subgroups of $\cG$, cf.\ 
\cite{SFB97}. In view of \cite[(1.5)]{SFB97}, a closed embedding $\cH \hookrightarrow \cG$ yields a closed embedding $V(\cH) \hookrightarrow V(\cG)$.
It now follows from \cite[(3.8)]{FPe05}, that the canonical map
\[ \iota_\ast : \msP(\cG_1) \lra \msP(\cG)\]
is injective, while \cite[(4.11)]{FPe05} implies that $\iota_\ast$ is closed. Owing to \cite[(4.11)]{FPe05}, \cite[(6.8)]{SFB97(2)} and \cite[(1.6)]{SFB97}
we have
\[ \dim \Proj(V(\cG))= \dim \msP(\cG) \ \text{and} \ \dim V(\cG_1)=\dim V(\fg),\]
so that Lemma \ref{IG1} yields $\dim V(\cG)=1$. Now \cite[(2.7)]{FV00} provides a decomposition $\cG \cong \cU\!\rtimes\!\GG_{m(r)}$, where $\cU$ is a 
V-uniserial normal subgroup of $\cG$ and $r\ge 0$. Since the multiplicative center $\cM(\cU\!\rtimes\!\GG_{m(r)})$ is assumed to be trivial, the group $\GG_{m(r)}$ 
acts faithfully on $\cU$.

We therefore assume that $\cG_1 \cong \SL(2)_1$, so that $\cC(\cG_1)=e_k$. We consider the centralizer $\cC_\cG(\cG_1)$ of $\cG_1$ in $\cG$. 
Since $\cC_\cG(\cG_1)_1 = \cC(\cG_1)=e_k$, it follows that $\cC_\cG(\cG_1)=e_k$. The adjoint representation thus provides a closed embedding
\[ \varrho : \cG \hookrightarrow \mathcal{AUT}(\fsl(2)) \cong \PSL(2)\]
from $\cG$ into the automorphism scheme $\mathcal{AUT}(\fsl(2))$. As $\cG$ is infinitesimal, $\varrho$ factors through a suitable Frobenius kernel of
$\PSL(2)$. Since $\PSL(2)$ is a factor of $\SL(2)$ by an \'etale normal subgroup, the Frobenius kernels of $\SL(2)$ and $\PSL(2)$ are isomorphic, so 
that there exists a closed embedding
\[\cG \hookrightarrow \SL(2)_s,\]
where $s=\height(\cG)$ is the height of $\cG$. Note that $\cG_1 \subseteq \SL(2)_1$, while $\Lie(\cG)=\fsl(2)$. This implies that $\SL(2)_1 \subseteq \cG$. 
Hence we may assume that $\cG$ is a subgroup of $\SL(2)$ such that $\cG_1=\SL(2)_1$.

We denote by $F : \SL(2) \lra \SL(2)$ the Frobenius endomorphism of $\SL(2)$. Note that $F$ induces an embedding $\bar{F} : \cG/\cG_1 \hookrightarrow 
\SL(2)$. Let $\cE \subseteq \cG/\cG_1$ be elementary abelian. Since $\cG/\cG_1 \hookrightarrow \cG^{(p)}$ has $p$-rank $\le 1$, Lemma \ref{IG1} shows that $\cE$ 
has height $\le 1$. Hence there is $g \in \SL(2)(k)$ such that  $g\bar{F}(\cE)g^{-1} \subseteq U_1$, the first Frobenius kernel of the group $U \cong \GG_a$ of strictly upper 
triangular unipotent matrices. Since $g=F(h)$ for some $h \in \SL(2)(k)$, passage to $h\cG h^{-1}$ allows us to assume that $\cE \subseteq U_2/U_1$. Let $\pi : \cG \lra \cG/\cG_1$ 
be the canonical projection and put $\cH:= \pi^{-1}(\cE)$. Then $\height(\cH)\le 2$, and we have $\SL(2)_1 \subseteq \cH \subseteq \SL(2)_1U_2 \cap \cG$. Thus, if $\cE \ne e_k$, then $\cH =
\SL(2)_1U_2$, whence $\SL(2)_1U_2 \subseteq \cG$.
This implies 
\[ \msrke(\cG) \ge \msrke(\SL(2)_1U_2) \ge \msrke(U_2) \ge 2,\]
a contradiction. Consequently, $\cE=e_k$, so that  \cite[(IV,\S3,(3.7)]{DG70} ensures that $\cG/\cG_1$ is diagonalizable. According to \cite[(5.4)]{FV03} this implies that 
\[ \cG= \SL(2)_1T_r\]
for $r=\height(\cG)$.\end{proof} 

\bigskip

\begin{Remark} The V-uniserial group schemes were classified in \cite{FRV01}. \end{Remark}

\bigskip
\noindent
Given a finite group scheme $\cG$, we let $\cG_{\rm lr}$ be the largest linearly reductive normal subgroup of $\cG$, see \cite[(I.2.37)]{Vo}. If $\cG$ is a reduced finite group scheme,
then $\cG_{\rm lr}$ corresponds to the largest normal subgroup $O_{p'}(\cG(k))$ of the finite group $\cG(k)$, whose order is prime to $p$.

Being a characteristic subgroup of $\cG^0$, the multiplicative center $\cM(\cG^0)$ is a normal subgroup of $\cG$.  Hence $\cG_{\rm red}$ acts on $\cG^0/\cM(\cG^0)$, and we put
\[ \cC_\cG := {\rm Cent}_{\cG_{\rm red}}(\cG^0/\cM(\cG^0)).\]
In the sequel, we let $\cW_n$ be the group scheme of Witt vectors of length $n$, cf.\ \cite[(V,\S1,${\rm n}^{\rm o}$1.6)]{DG70}.

To a finite subgroup scheme $\cG \subseteq \SL(2)$, we associate the group $\PP(\cG):=\cG/(\cG\cap\cC(\SL(2)))$. A linearly reductive subgroup scheme $\cG \subseteq \SL(2)$
will be referred to as a \textit{binary polyhedral group scheme}, see \cite{Fa06} for a classification of these groups. 

An associative algebra $\Lambda$ is said to be \textit{representation-finite}, provided there are only finitely many isoclasses of indecomposable $\Lambda$-modules.  We say that $\Lambda$
is \textit{domestic}, provided $\lambda$ is not representation-finite, and there exist $(\Lambda,k[T])$-bimodules $X_1,\ldots, X_m$ that are free of finite rank over $k[T]$ such that for every 
$d\ge 1$, all but finitely many isoclasses of $d$-dimensional indecomposable $\Lambda$-modules are of the form $[X_i\!\otimes_{k[T]}\!k[T]/(T\!-\!\lambda)^j]$, for some $i \in \{1,\ldots,m\}, j\in \NN$
and $\lambda \in k$, cf.\ \cite{Ri80}.
 
\bigskip

\begin{Theorem} \label{IG3} Suppose that $\cG$ is a finite group scheme such that $\cG_{\rm lr}=e_k$. If $\msrke(\cG)=1$, then one of the
following alternatives occurs:
\begin{enumerate}
\item[(a)] $\cG = \cG_{\rm red}$, and the finite group $\cG(k)$ has $p$-rank $1$ and $O_{p'}(\cG(k))=\{1\}$.
\item[(b)] There is a binary polyhedral group scheme $\tilde{\cG} \subseteq \SL(2)$ such that $\cG \cong \PP(\SL(2)_1\tilde{\cG})$.
\item[(c)] $\cG = \cU\!\rtimes\!\cG_{\rm red}$, where $\cU$ is V-uniserial of height $\height(\cU)\ge 2$ and $\cG(k)$ is cyclic and such that $p\!\nmid\!{\rm ord}(\cG(k))$.
\item[(d)] $\cG = ((\cW_n)_1\!\rtimes\!\GG_{m(r)})\!\rtimes\!\cG_{\rm red}$, where $\cG(k)$ is abelian $p\!\nmid\!{\rm ord}(\cG(k))$. \end{enumerate} \end{Theorem} 

\begin{proof} We first verify the following identity:
\[ (\ast) \ \ \ \ \ \ \msrke(\cG^0)\!+\!\msrke(\cG_{\rm red})=1.\]
By assumption, we have $\msrke(\cG^0) \le 1$ and $\msrke(\cG_{\rm red})\le1$. Suppose that both $p$-ranks are equal to $1$. Owing to \cite[(1.2)]{Fa06},
we have
\[ e_k = \cG_{\rm lr} = \cM(\cG^0)\!\rtimes\!O_{p'}(\cC_\cG).\]
Proposition \ref{IG2} now shows that $\cG^0 \cong \SL(2)_1T_r$ or $\dim V(\cG^0)=1$. In the latter case, the arguments of \cite[(3.1,2b)]{FV00} ensure the existence
of an elementary abelian subgroup $\cE \cong \GG_{a(1)}\!\times\!\ZZ/(p)$ of rank $2$, a contradiction. 

Hence we have $\cG_1 \cong \SL(2)_1$. Let $e_k \ne \cE \subseteq \cG_{\rm red}$ be elementary abelian. Then the $p$-elementary abelian group $\cE(k)$ acts on $\fsl(2)=\Lie(\cG)$. 
Since $\msrke(\cG)=1$, this action is faithful, so that $\cE(k) \subseteq \PSL(2)(k)$. Being a $p$-group, $\cE(k)$ is conjugate to a subgroup of strictly upper triangular matrices. 
Consequently, $V(\fsl(2))^{\cE(k)}\ne\{0\}$, so that there is a subgroup of type $\GG_{a(1)}\!\times\!\cE$ of rank $\ge 2$, a contradiction. Consequently,
\[ \msrke(\cG^0)\!+\!\msrke(\cG_{\rm red})=1,\]
as desired. \hfill $\diamond$ 

\smallskip
\noindent
If $\msrke(\cG^0)=0$, then $\cG^0$ is diagonalizable, so that $\cG^0 = \cM(\cG^0)=e_k$. Consequently,
\[ e_k = O_{p'}(\cC_\cG)=O_{p'}(\cG(k)),\]
and (a) holds.

Alternatively, identity ($\ast$) forces the group scheme $\cG_{\rm red}$ to be linearly reductive, implying $\cC_\cG=e_k$ and that $\cG_{\rm red}$ acts faithfully on 
$\cG^0$. By the same token, the group $\cG^0$ has $p$-rank $1$, while $\cM(\cG^0)=e_k$. Thanks to Proposition \ref{IG2}, we obtain
$\cG^0 \cong \SL(2)_1T_r$ or $\cG^0 \cong \cU\!\rtimes\!\GG_{m(r)}$, with $\GG_{m(r)}$ acting faithfully on $\cU$. 

Assuming $\cG^0=\SL(2)_1T_r$, it follows from \cite[(5.6)]{FV03} that the algebra $k\cG^0$ has domestic representation type. According to \cite[(4.4)]{Fa15}, $k\cG$ is 
representation-finite or domestic. In view of \cite[(2.7),(3.1)]{FV00}, the former alternative does not occur. We may now apply \cite[(4.7)]{Fa15} to see that $\cG$ is of type (b). 

We finally consider the case, where $\cG^0=\cU\!\rtimes\!\GG_{m(r)}$, where $\cU$ is V-uniserial. Note that the finite group $G:=\cG(k)$ acts on the unipotent radical $\cU$ of $\cG^0$ (cf.\ 
\cite[(IV,\S2,(3.3)]{DG70}) and hence on $k\cU$. However, $k\cU \cong k[X]/(X^{p^n})$ is uniserial, whence $\dim_k\Rad^i(k\cU)/\Rad^{i+1}(k\cU) = 1$. Consequently, $k\cU$ is a direct sum 
of one-dimensional $G$-modules, and there exists a group homomorphism $\zeta : G \lra k^\times$ such that $k\cU = k[x]$, where $x \in k\cU_\zeta$ is a weight vector such that $x^{p^n}=0$. 
This yields
\[ (\ast\ast) \ \ \ \ k\cU = \bigoplus_{i=0}^{p^n-1} k\cU_{i\zeta},\]
so that $\ker \zeta$ acts trivially on $k\cU$ and hence on $\cU$.

The group $G$ also acts on $\GG_{m(r)} \cong \cG^0/\cU$. There results an action of $G$ on the character group $X(\GG_{m(r)}) \cong \ZZ/(p^r)$ via automorphisms.This implies that $(G,G)$ acts 
trivially on $X(\GG_{m(r)})$ and thus on $\cG^0/\cU$. Observe that $\Rad(k\cG^0) = k\cG^0\Rad(k\cU)$ is the radical of $k\cG^0$, whose factor algebra is isomorphic to
$k(\cG^0/\cU)$, cf.\ \cite[(1.23)]{Fa10}. Since $\Rad(k\cG^0)^i = k\cG^0\Rad(k\cU)^i$, it follows that the multiplication induces surjections
\[ k(\cG^0/\cU)\!\otimes_k\!(\Rad(k\cU)^i/\Rad(k\cU)^{i+1}) \lra \Rad(k\cG^0)^i/\Rad(k\cG^0)^{i+1}\]
of $G$-modules, so that ($\ast\ast$) implies that the right-hand spaces are trivial $(G,G)$-modules. Since $(G,G)$ is linearly reductive, we conclude that $(G,G)$ acts trivially on $k\cG^0$ and hence on
$\cG^0$. As $G$ also acts faithfully on $\cG^0$, the group $G$ is abelian. 

If $\cU$ has height $\height(\cU)\ge 2$, then \cite[(3.1)]{FV99} implies that $\GG_{m(r)}$ acts trivially on $\cU$. Thus, $\cG^0=\cU$, so that ($\ast\ast$) yields $\ker \zeta =\{1\}$. Consequently,
$G$ is a subgroup of $k^\times$ and hence cyclic. This shows that (c) holds. 

If $\cU$ has height $1$, then \cite[(3.2)]{FV99} yields $\cU \cong (\cW_n)_1$ for some $n$, so that $\cG$ is of type (d). \end{proof}

\bigskip
\noindent
As a by-product, we obtain the following characterization of finite group schemes of finite- or domestic representation type.

\bigskip

\begin{Corollary} \label{IG4} Let $\cG$ be a finite group scheme. Then the following statements are equivalent:
\begin{enumerate}
\item $\cG$ is representation-finite or domestic.
\item $\msrke(\cG/\cG_{\rm lr})\!\le\!1$. \end{enumerate} \end{Corollary}

\begin{proof} Let $\cG':=\cG/\cG_{\rm lr}$. According to \cite[(1.1)]{Fa06}, there is an isomorphism $\cB_0(\cG) \cong \cB_0(\cG')$ between the principal blocks of $k\cG$ and $k\cG'$. A twofold 
application of \cite[(3.1)]{FV00} and \cite[(4.7)]{Fa15} now implies that $\cG$ is representation-finite or domestic if and only if $\cG'$ enjoys this property. The assertion now follows
from Theorem \ref{IG3} in conjunction with \cite[(2.7),(3.1)]{FV00} and \cite[(4.7)]{Fa15}. \end{proof}

\end{document}